\documentclass[letterpaper,12pt]{amsart}
\usepackage[margin=1.19in]{geometry}
\usepackage{graphicx, epigraph}
\graphicspath{ {./images/} }
\usepackage{amsmath,amsthm,amssymb,mathtools,amsfonts,mathrsfs}
\usepackage{xspace,xcolor}
\usepackage[breaklinks,colorlinks,citecolor=teal,linkcolor=teal,urlcolor=teal,pagebackref,hyperindex]{hyperref}
\usepackage[alphabetic]{amsrefs}
\usepackage{tikz-cd}
\usepackage[all]{xy}
\usepackage{url}
\usepackage[shortlabels]{enumitem}

\newtheorem{thm}{Theorem}[section]
\newtheorem{prop}[thm]{Proposition}

\newtheorem{lem}[thm]{Lemma}

\theoremstyle{definition}
\newtheorem{defn}[thm]{Definition}
\newtheorem{ex}[thm]{Example}
\newtheorem{rmk}[thm]{Remark}

\newtheorem*{nota*}{Notation}

\newcommand{\C}{\mathbb C}

\newcommand{\Z}{\mathbb Z}

\newcommand{\PP}{\mathbb P}
\newcommand\cC{\mathcal C}
\newcommand{\vir}{\mathrm{vir}}
\newcommand{\bM}{\overline{\mathcal M}}
\newcommand{\bA}{\boldsymbol{A}}

\title[operator formalism for GW-invariants of cap and tube]
{A new approach to the operator formalism for Gromov-Witten invariants of the cap and tube}

\author{Ajith Urundolil Kumaran}
\email{au270@cam.ac.uk}
\author{Longting Wu}
\email{wult@sustech.edu.cn}

\begin{document}

\maketitle

\begin{abstract}
Based on Johnson's operator formula for the equivariant Gromov-Witten theory of $\PP^1$-orbifolds, we give a new approach to the operator formalism by Okounkov and Pandharipande regarding the $\C^*$-equivariant Gromov-Witten theory of $\PP^1$ relative to one or two points. We also
extend their operator formalism in the non-equivariant specialization to the case where we allow negative contact orders.
\end{abstract}


\section{Introduction}
\subsection{Overview}

Let $X$ be a smooth projective variety and $D$ be a smooth divisor in $X$. The relative Gromov-Witten (GW-) theory of the pair $(X,D)$, roughly speaking, (virtually) counts maps from domain curves to the target $X$ with certain tangency conditions along the divisor $D$. The theory was founded in \cites{LR, IP2, EGH} on the symplectic side and in \cites{Jun1,Jun2} on the algebraic side which plays an important role in the development of GW-theory.

The recent works \cites{ACW,JPPZ,JPPZ2,TY,FWY,FWY2} provide a new perspective on relative GW-theory. The new point of view can be summarized in one sentence: \emph{Relative GW-theory of the pair $(X,D)$ can be treated as a certain large $r$ limit of the orbifold GW-theory of the root stack $X_{D,r}$.} Under the guidance of this new perspective, the relative GW-theory has made great progress. For example, in \cites{FWY,FWY2}, relative GW-theory was enlarged to include negative contact and the enlarged theory was shown to form a partial CohFT (cohomological field theory). In addition, relative quantum cohomology rings and Givental formalism
were also constructed using the genus-zero invariants of the enlarged theory.

But in general, it is quite hard to explicitly calculate the large $r$ limit of the orbifold GW-theory of the root stack $X_{D,r}$, as it involves rather complicated combinatorics. This limits the application of the new perspective. In this paper, we want to make a first step towards explicit computations by considering the simplest pairings $(X,D)$, i.e., cap and tube. Here cap stands for $\PP^1$ relative to one point and tube stands for $\PP^1$ relative to two points. Even in these simple cases, the
explicit calculations of the orbifold GW-theory of $X_{D,r}$ in the large $r$ limit already give us some nontrivial results. For example, it provides a new approach to the operator formalism\footnote{The operator formalism is expressed as a vacuum expectation on the infinite wedge space.} for the $\C^*$-equivariant GW-theory of the cap in \cite{OP3}, which plays an important role in the proof of Virasoro constraints for target curves. The original approach by Okounkov and Pandharipande used a complicated induction while the new approach here is more straightforward. The new approach can also
be used to extend the operator formula for the non-equivariant tube theory in \cite{OP1} to include negative contact which can be further used to deduce an explicit formula for the one-point function of the tube with negative contact.

Our new approach to the cap and tube theories relies on some knowledge of the $\C^*$-equivariant GW-theory of the root stack $\mathcal{C}_{r,s}$. Here $\mathcal{C}_{r,s}$ is obtained by successively applying the $r$-th root and $s$-th root constructions along the divisors $0$ and $\infty$ respectively. Fortunately, the equivariant GW-theory of $\mathcal{C}_{r,s}$ has been studied in \cites{J1,J,MT1,MT2}. In particular, Paul Johnson in his PhD thesis \cite{J1} (see also \cite{J} for its publication version) gave an explicit operator formula for it. 
We then analyse the large $r,s$ behavior of Johnson's operator formula. After some efforts, it turns out that the operator formalism for the equivariant cap and tube theories naturally appear when we take suitable coefficients\footnote{In \cite{TY}, it was shown that orbifold GW-invariants of $X_{D,r}$ (without large ages) are polynomials of $r$ when $r$ become sufficiently large and the constant terms give the relative GW-invariants of $(X,D)$.} of Johnson's operator formula in the large $r,s$ limit. We also analyse Johnson's operator formula in the non-equivariant specialization. By taking the large $r,s$ limit,
we can extend the operator formalism for the non-equivariant tube theory in \cite{OP1} to include negative contact. 


The explicit calculations in this paper also provide a new perspective on some known results in the literature.
For example, in \cite{TY}, it was shown that for sufficiently large $r$ and $X$ a curve, stationary orbifold invariants of $X_{D,r}$ are independent of $r$. This result can be verified in the case of $\mathcal{C}_{r,s=1}$ using our explicit computation (see Remark \ref{rmk:indepofr}). In \cite{TY2}, it was shown that genus $g$ orbifold GW-invariants of $X_{D,r}$, as a polynomial of $r$ for sufficiently large $r$, has $r$-degree bounded by $\max\{0,2g-1\}$. This result can also be verified for $\mathcal{C}_{r,s=1}$ using our explicit computation (see Remark \ref{rmk:r-degbd}). Actually, in our special case, the computation in Section \ref{sec:opeqv} shows that the coefficients of such $r$-polynomials can be determined from the Hodge integrals in the moduli space of curves, which might be of independent interest.

In order to give a more detailed account of our results, we need to introduce some notations.

\subsection{Basic setup}
\subsubsection{The Chen-Ruan cohomology of $\cC_{r,s}$} First, we have an obvious $\mathbb{C}^*$-action on $\cC_{r,s}$ such that $0$ and $\infty$ are fixed points. It naturally induces a $\C^*$-action on the inertia stack $I\cC_{r,s}$ of $\cC_{r,s}$. The fixed locus of the corresponding coarse moduli $\underline{I}\cC_{r,s}$ consist of $r$ copies of $0$ and $s$ copies of $\infty$.

The $\C^*$-equivariant Chen-Ruan orbifold cohomology $H^*_{CR,\C^*}(\cC_{r,s},\C)$ plays an important role in the definition of the equivariant orbifold invariants of $\cC_{r,s}$. As a vector space, it is equal to the $\C^*$-equivariant cohomology $H^*_{\C^*}(\underline{I}\cC_{r,s},\C)$. But $H^*_{CR,\C^*}(\cC_{r,s},\C)$ has a shifted grading and deformed cup product which are not relevant in this paper. We recommend the readers to see \cite{AJR} for an introduction to Chen-Ruan cohomology.

The $\C^*$-equivariant cohomology of a point $H^*_{\C^*}(pt,\C)$ is equal to $\C[t]$ where $t$ is the first Chern class of the standard representation of $\C^*$. So $H^*_{CR,\C^*}(\cC_{r,s},\C)$ is canonically a $\C[t]$-module. The Atiyah-Bott localization formula tells us that after localizing the appropriate element of $\C[t]$, $H^*_{CR,\C^*}(\cC_{r,s})$ has a basis given by the classes $\boldsymbol{0}_{i/r}$, $\boldsymbol{\infty}_{j/s}$ ($0\leq i\leq r-1$, $0\leq j\leq s-1$) where $\boldsymbol{0}_{i/r}$ is the Poincar\'e dual of the $\mathbb{C}^*$-fixed point $0$ in the component of $\underline{I}\cC_{r,s}$ with age $\frac{i}{r}$, and  $\boldsymbol{\infty}_{j/s}$ is the Poincar\'e dual of the $\mathbb{C}^*$-fixed point $\infty$ in the component of $\underline{I}\cC_{r,s}$ with age $\frac{j}{s}$.

\subsubsection{Orbifold GW-invariants of $\cC_{r,s}$}
Let $\bM_{g,m}(\cC_{r,s},d)$ be the moduli space of stable maps to $\cC_{r,s}$ with genus $g,m$ markings and degree $d$. The natural $\C^*$-action on $\cC_{r,s}$  induces a $\C^*$-action on $\bM_{g,m}(\cC_{r,s},d)$. The $\C^*$-equivariant orbifold GW-invariant can be defined via an integration over the equivariant virtual cycle of the moduli space:
\begin{equation}\label{eqn:orbilineeq}
\left\langle\tau_{k_1}(\gamma_1)\dots\tau_{k_m}(\gamma_m)\right\rangle_{g,d}^{\cC_{r,s},\,\C^*}\coloneqq \int_{[\bM_{g,m}(\cC_{r,s},d)]^{\vir}_{\C^*}}\prod_{i=1}^m\bar{\psi}_i^{k_i}ev_i^*(\gamma_i)
\end{equation}
where $ev_i$ denotes the evaluation map of the $i$-th marking which allows us to pull back the equivariant Chen-Ruan cohomology classes $\gamma_i\in H^*_{CR,\C^*}(\cC_{r,s})$,  $\bar{\psi}_i$ are psi-classes pulled back from the corresponding psi-classes of $\bM_{g, m}(\mathbb{P}^1,d)$ via the morphism $f: \bM_{g,m} (\cC_{r,s},d)\to \bM_{g,m}(\mathbb{P}^1,d)$ which forgets the orbifold structure. We add a bullet to denote the corresponding disconnected invariant, that is, $\langle\dots\rangle^{\bullet}$.

\begin{rmk}
We notice that the above notation \eqref{eqn:orbilineeq} does not specify the age information of each marking. Such age information can be specified via the choosing of $\gamma_i$. For example, if we choose $\gamma_i$ to be $\boldsymbol{0}_{\mu_i/r}$ ($\mu_i> 0$), then we know that the $i$-th marking is mapped to $0$ with age $\frac{\mu_i}{r}$.
\end{rmk}

\subsubsection{$\tau$ function}
The $\tau$ function encodes all the disconnected GW-invariants of $\cC_{r,s}$. It is defined as 
\begin{equation}\label{eqn:potential}
\tau=\sum_{g\in \Z}\sum_{d\geq 0}u^{2g-2}q^{d}\left\langle\exp\left(\sum_{k,i} x_k(i)\tau_k(\boldsymbol{0}_{i/r})+\sum_{k,j}x_k^*(j)\tau_k(\boldsymbol{\infty}_{j/s}) \right)\right\rangle_{g,d}^{\bullet\,\cC_{r,s},\,\C^*}.
\end{equation}
Here the formal variables $\{x_{k}(i)\}$, $\{x_{k}^*(j)\}$ are introduced so as to keep track of insertions $\tau_k(\boldsymbol{0}_{i/r})$, $\tau_k(\boldsymbol{\infty}_{j/s})$ respectively.

\subsubsection{Johnson's operator formula}
Using the virtual localization and orbifold ELSV formula, Johnson expressed the $\tau$ function as a vacuum expectation on the infinite wedge space:
\begin{thm}[\cites{J1,J}]
\begin{equation}\label{eqn:orbiP1_intro}
    \tau=\left \langle e^{\sum_{k,i} x_k(i) \boldsymbol{A}_{i/r}[k]} e^{\frac{t \alpha_r}{ur}} \left(\frac{q}{t^{1/r}(-t)^{1/s}}\right)^{H} e^{\frac{-t\alpha_{-s}}{us}} e^{\sum_{k,j} x^*_k(j)\boldsymbol{A}^*_{j/s}[k]}\right\rangle
\end{equation}
An explanation of the infinite wedge space and operators appearing in the above formula will be given in Section \ref{sec:infwedg}.
\end{thm}

\subsubsection{GW-invariants of the cap and tube}
Let $\bM_{g,n}(\PP^1/0,d,\vec{\mu}_0)$ be the moduli space which parametrizes relative stable maps from genus $g, n$-pointed curves to $\PP^1$ of degree $d$ with monodromy at $0$ given by an ordered partition $\vec{\mu}_0=(\mu_1,\cdots,\mu_\rho)$ of $d$. The GW-invariants of the pair $(\PP^1,0)$ can be defined as integrals against the virtual cycle as before:
\[\left\langle\vec{\mu}_0\left|\prod_{i=1}^n\tau_{k_i}(\gamma_i)\right.\right\rangle\coloneqq \int_{[\bM_{g,n}(\PP^1/0,d,\vec{\mu}_0)]^{\vir}}\prod_{i=1}^n\psi_i^{k_i}ev_i^*(\gamma_i)\]
where $\gamma_i\in H^*(\PP^1)$. As before, we add a bullet or $\C^*$ to denote the corresponding disconnected or $\C^*$-equivariant invariant. Here we omit the degree $d$ and genus $g$ in the bracket notation, since $d$ can be determined from $\vec{\mu}_0$ and $g$ can be determined from the dimension constraint.

For the pair $(\PP^1,\infty)$, we write the monodromy data on the right side of the bracket:
\[\left\langle\left.\prod_{i=1}^n\tau_{k_i}(\gamma_i)\right |\vec{\mu}_{\infty}\right\rangle.\]
As for the pair $(\PP^1,0\cup\infty)$, we write the monodromy data of $0$ on the left and write that of $\infty$ on the right:
\[\left\langle\vec{\mu}_0\left|\prod_{i=1}^n\tau_{k_i}(\gamma_i)\right |\vec{\mu}_{\infty}\right\rangle.\]

According to \cites{ACW,JPPZ,JPPZ2,TY}, we can redefine GW-invariants of the cap and tube as certain limits of orbifold GW-invariants of $\mathcal{C}_{r,s}$ (see Section \ref{sec:rel/orbGW}). A little more precisely, by replacing each multiplicity $\mu_i\in\vec{\mu}_0$ (resp. $\mu'_i\in\vec{\mu}_{\infty}$) with an insertion $0_{\mu_i/r}$ (resp. $\infty_{\mu'_{i/s}}$) which stands for the identity element of the twisted sector over $0$ (resp. $\infty$) with age $\mu_i/r$ (resp. $\mu'_i/s$), we get a corresponding orbifold invariant of $\mathcal{C}_{r,s}$. For example, the GW-invariant of the tube above corresponds to an orbifold invariant
\[\left\langle\prod_{i}\tau_0(0_{\mu_i/r})\prod_{i=1}^n\tau_{k_i}(\gamma_i)\prod_i\tau_0(\infty_{\mu'_i/s})\right\rangle^{\mathcal{C}_{r,s}}_{g,d}.\]
Here $\gamma_i$ should be seen as cohomology classes in the untwisted sector. For the pair $(\PP^1,0)$ (resp. $(\PP^1,\infty)$), we get an orbifold invariant of $\mathcal{C}_{r,s=1}$ (resp. $\mathcal{C}_{r=1,s}$).

For sufficiently large $r,s$, the corresponding orbifold invariant will be a polynomial of $r,s$ and the constant term matches with the relative GW-invariant we defined above. Note that for the pair $(\PP^1,0)$ (resp. $(\PP^1,\infty)$), the polynomial will only depend on $r$ (resp. $s$). 

In \cites{FWY,FWY2}, relative GW-invariants were generalized to include negative contact using the orbifold perspective. Specializing to the previous cap or tube case, negative contact $\mu_i\in\vec{\mu}_0$ (resp. $\mu'_i\in\vec{\mu}_{\infty}$) simply corresponds to the insertion $0_{\frac{r+\mu_i}{r}}$ (resp. $\infty_{\frac{s+\mu_i'}{s}}$).


\subsection{Main results}
The orbifold perspective gives us a new approach to the operator formalism for GW-invariants of the cap and tube. For technical reasons, we only consider the following two cases. 

(I) In the first case, we consider equivariant invariants without negative contact.
We always treat the cap case carefully since the tube case follows by a similar analysis. So we want to compute the following equivariant GW-invariants of $(\PP^1,0)$:
\begin{equation*}
\left\langle\vec{\mu}\left |\prod_{i=1}^n\tau_{k_i}(\boldsymbol{0})\prod_{j=1}^m\tau_{l_j}(\boldsymbol{\infty})\right.\right\rangle^{\bullet,\mathbb{C}^*}
\end{equation*}
where $\boldsymbol{0}$, $\boldsymbol{\infty}$ are two natural equivariant classes in $H^*_{\C^*}(\PP^1)$. By varying $k_i$ and $l_j$, we get a generating function
\[G(\vec{\mu}|z_1,\cdots,z_n,w_1,\cdots,w_m).\]
The orbifold perspective leads us to consider the following orbifold invariants:
\begin{equation}\label{eqn:orbpers}
\left\langle\prod_i \tau_0(\boldsymbol{0}_{\mu_i/r})\prod_{i=1}^n\tau_{k_i}(\boldsymbol{0}_{0/r})\prod_{j=1}^m\tau_{l_j}(\boldsymbol{\infty}_{0/1})\right\rangle^{\bullet,\mathcal{C}_{r,s=1},\mathbb{C}^*}.
\end{equation}
By \cite{FWY2}, it is still true that the equivariant orbifold invariants above are polynomials of $r$ and the constant terms give the equivariant relative invariants we want. So we can use Johnson's operator formula \eqref{eqn:orbiP1_intro} to calculate 
the generating series of such orbifold invariants in the large $r$ limit. The explicit calculation will be done in Section \ref{sec:opeqv}. The generating series of \eqref{eqn:orbpers} in the large $r$ limit will be a function depends on $r$, after taking the $r^0$-coefficient we get the following operator formula:
\begin{thm}[=Theorem \ref{thm:opfm_cap0}]
\[G(\vec{\mu}|z_1,\cdots,z_n,w_1,\cdots,w_n)=\left\langle  \prod_{i}\frac{\alpha_{\mu_i}}{\mu_i}E(z_1,\cdots,z_n,t)e^{\alpha_{-1}}\prod_{j=1}^m \bA^*(w_j)\right\rangle.\]
\end{thm}
The explicit forms of the above operators can be found in Sections \ref{sec:infwedg} and \ref{sec:opeqv}. 

\begin{rmk}
The above equality holds only when we add the unstable contributions to the generating function $G(\vec{\mu}|z_1,\cdots,z_n,w_1,\cdots,w_n)$ similar as in \cite[Section 2.3.4]{OP3}.
\end{rmk}

\begin{rmk}
We can similarly calculate the generating function of the pair $(\PP^1,\infty)$. The resulting operator formula will recover the operator formula in \cite[Proposition 3.2]{OP3} (see also Remark \ref{rmk:OPfm}).
\end{rmk}


Similarly, by analysing Johnson's operator formula in the large $r,s$ limit, we have the following operator formula for the equivariant GW-invariants of the tube:
\begin{thm}[=Theorem \ref{thm:eqtubefm}]
\[G(\vec{\mu}|z_1,\cdots,z_n,w_1,\cdots,w_n|\vec{\nu})=\left\langle  \prod_{i}\frac{\alpha_{\mu_i}}{\mu_i}E(z_1,\cdots,z_n,t)E(w_1,\cdots,w_m,-t)\prod_{j}\frac{\alpha_{v_j}}{v_j}  \right\rangle.\]
\end{thm}

(II)  In the second case, we consider non-equivariant invariants with negative contact. As before, we start from the cap case. So we want to compute the following non-equivariant GW-invariant of $(\PP^1,0)$:
\begin{equation*}
\left\langle \vec{\mu}_0\left|\prod_{i=1}^n\tau_{k_i}(\omega)\right.\right\rangle^{\bullet}
\end{equation*}
where $\vec{\mu}_0=(a_1, \dots, a_{l_0}, b_1, \dots, b_{m_0})$ is an ordered partition of degree $d$ such that $a_i>0$, $b_j<0$, and $\omega$ is the point class of $\PP^1$. By definition of relative invariants with negative contact in Section \ref{sec:rel/orbGW}, we need to determine the large $r$ limit of the following orbifold invariant:
\begin{equation}\label{eqn:noneqnorb_intro}
r^{m_0}\left\langle \prod_{i=1}^{l_0}\tau_0(0_{a_i/r})\prod_{i=1}^{m_0}\tau_0(0_{(r+b_i)/r})\prod_{i=1}^n\tau_{k_i}(\omega)\right\rangle^{\bullet,\mathcal{C}_{r,s=1}}
\end{equation}
and take the $r^0$-coefficient. It is equivalent to consider the following equivariant orbifold invariant:
\begin{equation*}
r^{m_0}\left\langle \prod_{i=1}^{l_0}\tau_0(\boldsymbol{0}_{a_i/r})\prod_{i=1}^{m_0}\tau_0(\boldsymbol{0}_{(r+b_i)/r})\prod_{i=1}^n\tau_{k_i}(\boldsymbol{0}_{0/r})\right\rangle^{\bullet,\mathcal{C}_{r,s=1},\C^*}.
\end{equation*}
Again we apply Johnson's operator formula \eqref{eqn:orbiP1_intro} and analyse its large $r$ behavior. The explicit calculation will be done in Section \ref{sec:noneqopfm}. The final result is that
\begin{thm}[=Theorem \ref{thm:noneqcap}]
\begin{equation*}
   \left\langle\vec{\mu}_0\left|\prod_{i=1}^n \tau_{k_i}(\omega)\right.\right\rangle^{\bullet}= \frac{1}{(\prod_{i=1}^{l_0} a_i) d!} \left\langle \prod_{i=1}^{l_0} \alpha_{a_i} (\sum_{P} N_{\vec{b}_p} \mathcal{E}_{\vec{b}_P}[0] )\prod_{i=1}^n \mathcal{E}_0[k_i] (\alpha_{-1})^d\right\rangle
\end{equation*}
where we sum over all partitions $P$ of $\{1, ..., m_0\}$. 
\end{thm}
A more detailed explanation of the operators in the above formula can be found in Section \ref{sec:noneqopfm}.

Similarly, for the tube case, we have
\begin{thm}[=Theorem \ref{thm:opform}]
Let 
\[\vec{\mu}_0=(a_1, \cdots, a_{l_0}, b_1, \cdots, b_{m_0}),\quad \vec{\mu}_{\infty}=(a_1', \cdots, a'_{l_{\infty}}, b_1', \cdots, b'_{m_{\infty}})\]
such that $a_i,a'_i>0$, $b_j, b'_j<0$, $\forall i,j$. We have
\begin{gather*}
   \left\langle\vec{\mu}_0\left|\prod_{i=1}^n \tau_{k_i}(\omega)\right|\vec{\mu}_{\infty}\right\rangle^{\bullet}=\\ \frac{1}{\prod_{i=1}^{l_0} a_i \prod_{i=1}^{l_{\infty}} a_i'} \left\langle \prod_{i=1}^{l_0} \alpha_{a_i} (\sum_{P} N_{\vec{b}_p} \mathcal{E}_{\vec{b}_P}[0] ) \prod_{i=1}^n \mathcal{E}_0[k_i] (\sum_{P'} N_{\vec{b}'_{P'}} \mathcal{E}_{\vec{b}'_{P'}}^*[0]) \prod_{i=1}^{l_{\infty}} \alpha_{-a'_i}\right\rangle
\end{gather*}
where we sum over all partitions $P$ of $\{1, ..., m_0\}$ and all partitions $P'$ of $\{1, ..., m_{\infty}\}$.
\end{thm}
Using the above operator formula for the tube, we deduce an explicit formula for the one-point function of the tube with negative contact in Section \ref{sec:Ex}.

\begin{rmk}
The operator formula of the tube will specialize to that of cap (up to the factor $d!$) once we set $\vec{\mu}_{\infty}=(\underbrace{1,\cdots,1}_d)$. If there are no negative contact orders, then it will  specialize to the operator formula of the tube in \cite{OP1}. 
The new thing here is the operators of the form
\[\sum_{P} N_{\vec{b}_p} \mathcal{E}_{\vec{b}_P}[0]\]
which account for negative contact (see Definition \ref{def:extranota2} for more details on notations).
\end{rmk}

\subsection{Organization of the paper}
In Section \ref{sec:rel/orbGW}, we give a brief review of relative and orbifold GW-theories. In Section \ref{sec:infwedg}, we give an introduction to the infinite wedge formalism. In Section \ref{sec:opfm}, we give an analysis of large $r,s$ behavior of Johnson's operator formula and prove our main results. In Section \ref{sec:Ex}, we give an explicit formula for the one-point function of the tube with negative contact. The Appendices \ref{sec:interaction} and \ref{sec:wdvvcal} give some technical results needed in Section \ref{sec:opfm}.

\subsection{Acknowledgements}
Part of the results of this paper grew out of work done in the first author's Master Thesis which was supervised by Honglu Fan, the second author and Rahul Pandharipande. The first author is deeply indebted to Honglu Fan and Rahul Pandharipande for their guidance. We both thank Honglu Fan, Rahul Pandharipande, Georg Oberdieck and Dhruv Ranganathan for their helpful conversations and comments. The second author was supported by grant ERC-2017-AdG-786580-MACI. The second author was also grateful to Max Planck Institute for Mathematics in Bonn for its hospitality and financial support.

This project has received funding from the European Research Council (ERC) under the European Union’s
Horizon 2020 research and innovation program (grant agreement No. 786580).

\section{Relative GW-Theory vs. Orbifold GW-theory}\label{sec:rel/orbGW}
\subsection{General Theory}
In this section we recall the basic notions of relative GW-theory with possibly negative contact orders. The content and notation in this section is largely based on \cites{FWY,FWY2}. 
\begin{defn}A topological type $\Gamma$ is a tuple $(g,n,\beta, \rho, \vec{\mu})$ where $g,n$ are non-negative integers, $\beta\in H_2(X, \mathbb{Z})$ is a curve class and $\vec{\mu}\in (\mathbb{Z}^*)^{\rho}$ is a partition (we allow negative integers) of the intersection number $\int_{\beta} D$. 
\end{defn}
Let $X_{D,r}$ be the $r$-th root stack along the divisor $D$. We use $\underline{I}(X_{D,r})$ to denote the coarse moduli space of the inertia stack of $X_{D,r}$. It has the following form:
\begin{equation}
    \underline{I}(X_{D,r})\cong X \sqcup \bigsqcup_{i=1}^{r-1} D.
\end{equation}
We label the components of this inertia stack by $\frac{i}{r}$ for $i=0,...,r-1$. These fractional integers are called ages. The age $\frac{0}{r}$ is reserved for the component isomorphic to $X$. The other components are called twisted sectors. A topological type $\Gamma=(g, n, \beta, \rho, \vec{\mu})$ determines a moduli space of stable maps into $X_{D,r}$ in the following way:
\begin{itemize}
    \item $g, \beta$ correspond to genus and curve class.
    \item $n$ is the number of markings without orbifold structure.
    \item $\rho$ is the number of markings with orbifold structure.
    \item When $\mu_i > 0$, the evaluation map of the corresponding marking lands on the twisted sector with age $\frac{\mu_i}{r}$.
    \item When $\mu_i < 0$, the evaluation map of the corresponding marking lands on the twisted sector with age $\frac{r + \mu_i}{r}$.
\end{itemize}
Note that for this to make sense we need $r> \max_{1\leq i \leq \rho} \lvert \mu_i \rvert$. We denote this moduli space by $\bM_{\Gamma}(X_{D,r})$. The domain curves of these stable maps are connected. If we also allow disconnected domain curves we will write $\bM^{\bullet}_{\Gamma}(X_{D,r})$. In this case $g$ in $\Gamma$ can be negative.  Let $ev_{X,i}:\bM_{\Gamma}(X_{D,r})\to X$ denote the evaluation map at the $i$-th marked point without orbifold structure for $i=1, ..., n$. Similarly let $ev_{D,j}:\bM_{\Gamma}(X_{D,r})\to D$ denote the evaluation map at the $j$-th marked point with orbifold structure for $j=1, .., \rho$. We consider the forgetful map
\begin{equation}
    \tau:\bM_{\Gamma } (X_{D,r}) \to \bM_{g, n+ \rho}(X, \beta) \times_{X^\rho} D^\rho.
\end{equation}
And we define the psi-class $\bar{\psi}_i:=\tau^* \psi_i$ for $i=1,..., n+\rho$.
\begin{defn}
\label{def:orbifoldinv}
Let $\Gamma=(g, n, \beta, \rho, \vec{\mu})$ be a topological type. Let $\underline{\alpha}=(\bar{\psi}^{a_1}\alpha_1, ..., \bar{\psi}^{a_n} \alpha_n)\in (\mathbb{C}[\bar{\psi}]\otimes H^*(X))^n$. Let $\underline{\epsilon}=(\bar{\psi}^{b_1}\epsilon_1, ..., \bar{\psi}^{b_\rho}\epsilon_{\rho})\in (\mathbb{C}[\bar{\psi}]\otimes H^*(D))^{\rho}$. We define
\begin{equation}
    \langle \underline{\epsilon},  \underline{\alpha} \rangle_{\Gamma}^{X_{D,r}} \coloneqq \int_{[\bM_{\Gamma}(X_{D,r})]^{\vir}} \prod_{j=1}^{\rho} \bar{\psi}_{D,j}^{b_j} ev_{D,j}^*(\epsilon_j) \prod_{i=1}^{n} \bar{\psi}_{X,i}^{a_i} ev_{X,i}^*(\alpha_i),
\end{equation}
where $\bar{\psi}_{D,j}, \bar{\psi}_{X,i}$ are psi-classes of the corresponding markings.
If we replace $\bM_{\Gamma}(X_{D,r})$ with $\bM^{\bullet}_{\Gamma}(X_{D,r})$ we get the disconnected invariant which we denote by $\langle \underline{\epsilon},  \underline{\alpha} \rangle_{\Gamma}^{\bullet X_{D,r}}$.
\end{defn}
Following \cites{FWY2}, we can now define relative invariants with possibly negative contact orders as a limit of orbifold invariants.
\begin{defn}
\label{def:relinvariant}
We define
\begin{equation}
    \langle \underline{\epsilon},  \underline{\alpha} \rangle_{\Gamma}^{(X,D)}\coloneqq \left [\lim_{r\to \infty} r^{\rho_{-}} \langle \underline{\epsilon},  \underline{\alpha} \rangle_{\Gamma}^{X_{D,r}}\right]_{r^0},
\end{equation}
where $[\cdot]_{r^0}$ extracts the constant coefficient of a polynomial in $r$ and $\rho_-$ is the number of negative integers in $\vec{\mu}$. The disconnected relative invariant $\langle \underline{\epsilon},  \underline{\alpha} \rangle_{\Gamma}^{\bullet (X,D)}$ can be defined as the limit of the corresponding disconnected orbifold invariants.
\end{defn}
\begin{rmk}
Definition~\ref{def:relinvariant} makes sense since in \cite{FWY2}, it is shown that for large enough $r$ the expression inside the limit becomes a polynomial in $r$. Moreover it is shown in~\cite{TY} that if there are no negative contact orders, then this definition agrees with the definition of relative invariants in the traditional sense \cites{LR, IP2, EGH, Jun1, Jun2}. In \cite{FWY2}, relative invariants with negative contact are also expressed using relative invariants in the traditional sense. But we will not use the second definition in this paper.
\end{rmk}
\begin{rmk}
\label{rem:genus0rindep}
If the genus $g$ equals to $0$ in Definition~\ref{def:relinvariant} the expression inside the limit becomes independent of $r$ for $r$ sufficiently large. This is Theorem 3.2 in \cite{FWY}.
\end{rmk}
We note that the disconnected relative invariant can be expressed as a sum of products of connected relative invariants. This just follows from the fact that the same holds for orbifold invariants. More precisely, we have

\begin{equation}\label{eqn:prodform}
    \langle \underline{\epsilon}, \underline{\alpha} \rangle_{\Gamma}^{\bullet (X,D)}=\sum_{\Gamma'=\{\Gamma_\pi\}} \prod_{\pi} \langle \underline{\epsilon}^\pi, \underline{\alpha}^\pi \rangle_{\Gamma_\pi}^{(X,D)}
\end{equation}
where $\Gamma'$ are finite collections of topological types $\Gamma_\pi$ where the total genus is $g$ and total degree is $\beta$ and all the interior markings and relative markings of $\Gamma$ are distributed over these $\Gamma_\pi$. The notation $\underline{\epsilon}^\pi$ and $\underline{\alpha}^{\pi}$ just mean that we extract a subvector of $\underline{\epsilon}$ resp. $\underline{\alpha}$ according to $\Gamma_\pi$.

\subsection{Specializing to orbifold lines}
In this paper, we focus on the case that $X$ is $\PP^1$ and $D$ is the union of two points $0$, $\infty$ in $\PP^1$. In this case, the curve class $\beta$ can be replaced by a non-negative integer $d$. We also need to specify the orbifold markings over $0$ or $\infty$. So the topological type $\Gamma$ is changed to be 
\[(g,n,d,\rho^0, \rho^{\infty},\vec{\mu}_0,\vec{\mu}_{\infty})\]
where the new notations $\rho^0,\rho^{\infty}$ indicate the number of markings with orbifold structure over $0,\infty$ respectively, and $\vec{\mu}_0\in (\mathbb{Z}^*)^{\rho^0} ,\vec{\mu}_{\infty}\in (\mathbb{Z}^*)^{\rho^{\infty}}$ are the corresponding partitions of $d$. We still use 
\begin{equation}\label{eqn:orbiline}
\langle \underline{\epsilon},  \underline{\alpha} \rangle_{\Gamma}^{\cC_{r,s}}
\end{equation}
to denote the corresponding orbifold GW-invariants of $\cC_{r,s}$ (see Definition \ref{def:orbifoldinv}). Note that here the insertion $\underline{\epsilon}$ can be written as a union of insertions \[\underline{\epsilon_0}\in \left(\mathbb{C}[\bar{\psi}]\otimes H^*(\text{point})\right)^{\rho^0},\quad \underline{\epsilon_{\infty}}\in \left(\mathbb{C}[\bar{\psi}]\otimes H^*(\text{point})\right)^{\rho^{\infty}} \]
corresponding to markings over $0,\infty$ respectively. 

For large enough $r$ and $s$, it is still true\footnote{Since $0$ and $\infty$ are two distinct points, the argument of \cite{FWY2} still works.} that 
\[r^{\rho^0_-}s^{\rho^{\infty}_-}\langle \underline{\epsilon},  \underline{\alpha} \rangle_{\Gamma}^{\cC_{r,s}}\]
depends polynomially on $r$ and $s$, where $\rho^0_-,\rho^{\infty}_-$ are the number of negative integers in $\vec{\mu}_0,\vec{\mu}_{\infty}$ respectively.
So we can still extract the constant term and get the definition for the relative invariant of the pair $(\PP^1,\{0,\infty\})$ with possibly negative contact:
\[\langle \underline{\epsilon},  \underline{\alpha} \rangle_{\Gamma}^{(\PP^1,\{0,\infty\})}\coloneqq\left [\lim_{r,s\to \infty} r^{\rho^0_-}s^{\rho^{\infty}_-}\langle \underline{\epsilon},  \underline{\alpha} \rangle_{\Gamma}^{\cC_{r,s}}\right]_{r^0s^0}.\]
If we set $s=1$, then $\infty$ is no longer a stacky point. So $\rho^{\infty}=0$ and $\vec{\mu}_{\infty}$ is empty. For sufficiently large $r$, $r^{\rho^0_-}\langle \underline{\epsilon},  \underline{\alpha} \rangle_{\Gamma}^{\cC_{r,1}}$
becomes a polynomial in $r$. The constant term 
gives the definition for the relative invariant of $(\PP^1,0)$ with possibly negative contact:
\[\langle \underline{\epsilon},  \underline{\alpha} \rangle_{\Gamma}^{(\PP^1,0)}\coloneqq\left [\lim_{r\to \infty} r^{\rho^0_-}\langle \underline{\epsilon},  \underline{\alpha} \rangle_{\Gamma}^{\cC_{r,1}}\right]_{r^0}.\]
\begin{rmk}
\label{rem:notation}
If the insertion $\underline{\epsilon}$ is trivial, that is, $\underline{\epsilon}=(1,1,\cdots,1)$, then 
the relative invariant $\langle \underline{\epsilon},  \underline{\alpha} \rangle_{\Gamma}^{(\PP^1,\{0,\infty\})}$ with $\underline{\alpha}=(\bar{\psi}^{a_1}\alpha_1, ..., \bar{\psi}^{a_n} \alpha_n)$ could also be denoted as\footnote{When there are no negative contact orders, such notation has been used by Okounkov and Pandharipande in their famous trilogy \cites{OP1,OP2,OP3}.}
\[\left\langle\vec{\mu}_0\left|\prod_{i=1}^n\tau_{a_i}(\alpha_i)\right|\vec{\mu}_{\infty}\right\rangle.\]
Here we omit the notation $\Gamma$, because except for the genus the other data can be directly read from the above notation and, moreover, the genus can be determined from the dimension constraint. Similarly we use 
\[\left\langle\vec{\mu}_0\left |\prod_{i=1}^n\tau_{a_i}(\alpha_i)\right.\right\rangle,\quad \left\langle \left.\prod_{i=1}^n\tau_{a_i}(\alpha_i)\right|\vec{\mu}_{\infty}\right\rangle\]
to denote the relative invariant of the pair $(\PP^1,0)$, $(\PP^1,\infty)$ respectively when the insertion $\underline{\epsilon}$ is trivial.
\end{rmk}

\section{Orbifold GW-invariants via the infinite wedge}\label{sec:infwedg}
In this section, we want to give a brief introduction of Johnson's operator formula in the half-infinity wedge space for the $\tau$ function \eqref{eqn:potential}.

The half-infinity wedge space is deeply related with the representation theory of the symmetric group. In his PhD thesis \cite{J1}, Johnson applied the localization formula to the equivariant orbifold GW-potential \eqref{eqn:potential} to express it as a potential of Hurwitz-Hodge integrals. The orbifold ELSV formula \cite{JPT} expresses these integrals as double Hurwitz numbers. As these numbers are naturally connected with the character theory of the symmetric group, Johnson was able to express the equivariant orbifold potential as an operator expectation on the half-infinity wedge space. We now give a short introduction to the half-infinity wedge space.  This section is based on \cite{OP1}. 

Let $V$ be a vector space with basis $\{\underline{k}\}$ indexed by the half-integers:
\begin{equation*}
    V= \bigoplus_{k\in \mathbb{Z}+\frac{1}{2}} \mathbb{C} \underline{k}.
\end{equation*}
For each subset $S=\{s_1> s_2> ...\}\subset \mathbb{Z}+\frac{1}{2}$ satisfying:\\

(i) $S_{+}= S \setminus (\mathbb{Z}_{\leq0}-\frac{1}{2})$ is finite,\\

(ii) $S_{-}= (\mathbb{Z}_{\leq0}-\frac{1}{2}) \setminus S$ is finite,\\
\\
we denote by $v_S$ the following infinite  wedge product:
\begin{equation*}
    v_S=\underline{s_1} \wedge \underline{s_2} \wedge ...
\end{equation*}
Then we define 
\begin{equation*}
    \Lambda^{\frac{\infty}{2}} V \coloneqq \bigoplus \mathbb{C} v_S.
\end{equation*}
It is called the half-infinity wedge space. Let $(\cdot, \cdot)$ be the unique inner product on $\Lambda ^{\frac{\infty}{2}} V$ such that those $v_S$ form an orthonormal basis. We now define some important operators on this space.
\begin{defn}
       Let $k\in \mathbb{Z}+\frac{1}{2}$.
       We define the operator $\Psi_k$ by the equation 
       \begin{equation*}
       \Psi_k v := \underline{k} \wedge v ,
       \end{equation*}
       for all $v\in \Lambda ^{\frac{\infty}{2}} V$. 
\end{defn}

Let $\Psi_j^*$ be the adjoint of $\Psi_j$ with respect to the previously defined inner product.

\begin{defn}
    We define normally ordered products:
    \begin{equation*}
        :\Psi_i \Psi_j^*:=
        \begin{cases} 
        \Psi_i \Psi_j^*,\text{ if } j>0, \\
        -\Psi_j^* \Psi_i,\text{ if } j<0. \\
        \end{cases}
    \end{equation*}
    
\end{defn}
We now define a representation of the Lie algebra $\mathfrak{gl}(V)$ on $\Lambda ^{\frac{\infty}{2}} V$.
For $i,j\in \mathbb{Z}+\frac{1}{2}$, let $E_{ij}$ denote the standard basis element in $\mathfrak{gl}(V)$. We then define a projective representation by the following assignment:
\begin{equation*}
    E_{ij} \mapsto \ :\Psi_i \Psi_j^*:
\end{equation*}
\begin{defn}
    We define the charge operator $C$ to be the operator on $\Lambda ^{\frac{\infty}{2}} V$ corresponding to the identity matrix $\sum_{k\in \mathbb{Z}+\frac{1}{2}} E_{kk}$ via the previously defined representation.
\end{defn}
The kernel of $C$ is called the zero charge space and is spanned by the vectors 
\begin{equation*}
    v_\lambda=\lambda_1-\frac{1}{2} \wedge \lambda_2-\frac{3}{2} \wedge \lambda_3-\frac{5}{2} \wedge ...,
\end{equation*}
where $\lambda$ varies over partitions of nonnegative numbers. We denote the zero charge space by $\Lambda_0 ^{\frac{\infty}{2}} V$.
\begin{defn}
    We define the energy operator $H$ to be the operator on the zero charge space corresponding to the matrix $\sum_{k\in \mathbb{Z}+\frac{1}{2}} k E_{kk}$.
\end{defn}
One can check that we have the following equation:
\begin{equation*}
    H v_{\lambda} = \lvert \lambda \rvert v_{\lambda}.
\end{equation*}
We call the eigenspace of $H$ corresponding to a nonnegative integer $l$ the energy $l$ subspace. The vacuum vector
\begin{equation*}
    v_{\emptyset}=\underline{-\frac{1}{2}} \wedge \underline{-\frac{3}{2}} \wedge \underline{-\frac{5}{2}} \wedge ....,
\end{equation*}
is an eigenvector of $H$ with eigenvalue $0$. The \emph{vacuum expectation} of an operator $A$ on $\Lambda ^{\frac{\infty}{2}} V$ is defined by the following inner product:
\begin{equation*}
    \langle A \rangle := (A v_{\emptyset}, v_{\emptyset})
\end{equation*}
We now define an important Laurent series whose coefficients are operators on $\Lambda ^{\frac{\infty}{2}} V$.
\begin{defn}\label{def:E-op}
    For $j\in \mathbb{Z}$, we define
    \begin{equation*}
    \mathcal{E}_j(z)\coloneqq\sum_{k\in \mathbb{Z}+\frac{1}{2}} e^{z(k-\frac{j}{2})} E_{k-j,k} + \frac{\delta_{j,0}}{e^{\frac{z}{2}}-e^{-\frac{z}{2}}}.
    \end{equation*} 
\end{defn}
It can be shown that the we have the following equation:
\begin{equation*}
    \mathcal{E}_j(z)^*=\mathcal{E}_{-j}(z),
\end{equation*}
for all $j\in \mathbb{Z}$.
We now state an important commutation relation which will be used throughout this paper. First we define:
\begin{gather}
    \mathcal{S}(z)\coloneqq\frac{e^{\frac{z}{2}}-e^{-\frac{z}{2}}}{z}\label{eqn:S-fun}\\
    \varsigma (z)\coloneqq z \mathcal{S}(z) \label{eqn:vsig}
\end{gather}
Then we can show for $j$ and $k$ arbitrary integers that the following equation is true:
\begin{equation}
\label{eq:commrel}
    [\mathcal{E}_j(z), \mathcal{E}_k(w) ]= 
    \varsigma ( jw-kz) \mathcal{E}_{j+k}(z+w).
\end{equation}
\begin{defn}
For $j\neq 0$ we define
\begin{equation}
    \alpha_j\coloneqq \mathcal{E}_{j}(0).
\end{equation}
\end{defn}
From Eq. ~\eqref{eq:commrel} we get the following relations: 
For $j\neq0$ and $k\neq0$ we have
\begin{equation}
\label{eq:commrel1}
    [\alpha_j, \alpha_k]=j \delta_{j,-k}.
\end{equation}
We also have
\begin{equation}
\label{eq:commrel2}
    [\alpha_j, \mathcal{E}_k[0]]=
    \begin{cases}
    j \alpha_{j+k}, \text{ if } j+k\neq0, \\
    j C, \text{ if } j+k=0\\
    \end{cases}
\end{equation}
where $C$ is the charge operator.
\begin{defn}
Let $A$ be an operator on $\Lambda^{\frac{\infty}{2}} V$. Let $l$ be an integer. We say $A$ has energy $l$ if
\begin{equation*}
    [H,A]=l A.
\end{equation*}
\end{defn}

\begin{rmk}\label{rmk:energy}
Note that up to some constant, $H$ is just the $z$-coefficient of $\mathcal{E}_0(z)$. So by the commutation relation \eqref{eq:commrel}, we see that $\mathcal{E}_j(z)$ has energy $-j$. This implies all the coefficients of $\mathcal{E}_j(z)$ have energy $-j$. In particular, we know that $\alpha_j$ has energy $-j$.
\end{rmk}

\begin{rmk}
Johnson deals with a more general target orbicurve in his thesis. The orbifold $\mathcal{C}_{r,s}$ is what he calls effective which means that the generic point has trivial isotropy group. He works with a generalization of $\mathcal{C}_{r,s}$. Let us call it $\mathcal{X}$. This $\mathcal{X}$ is an orbifold whose underlying topological space is $\mathbb{P}^1$ with orbifold structure around $0$ and $\infty$ but the isotropy group of the generic point is allowed to be a nontrivial group $K$. Such an orbifold is called ineffective (see \cite[Chapter II]{J1}).
\end{rmk}
We now introduce important operators which are necessary to state a theorem from \cite{J1}. Due to the previous remark we will set $K$ to be the trivial group in all the statements and definitions we take from \cite{J1}. We recall the Pochhammer symbol:
\begin{equation}\label{eqn:Psymb}
    (1+x)_n\coloneqq
    \begin{cases}
    (x+1) ... (x+n),\, n\geq0, \\
    (x(x-1) ... (x+n+1))^{-1},\, n\leq 0.
    \end{cases}
\end{equation}
where $n$ is an integer.
\begin{defn}\label{def:A-ops}
    Let $a$ and $r$ be integers with $0\leq a \leq r-1$. If $a=0$ we define the following power series:
    \begin{equation*}
        \boldsymbol{A}_{0/r}(z,u,t)\coloneqq\frac{1}{u} \mathcal{S}(ruz)^{\frac{tz}{r}} \sum_{i=-\infty}^{\infty} \frac{(tz\mathcal{S}(ruz))^i}{(1+\frac{tz}{r})_i} \mathcal{E}_{ir}(uz).
    \end{equation*}
    And if $a>0$ we define 
    \begin{equation*}
        \boldsymbol{A}_{a/r}(z,u,t)\coloneqq\frac{t^{\frac{a}{r}}}{u}\frac{z}{tz+a} \mathcal{S}(ruz)^{\frac{tz+a}{r}} \sum_{i=-\infty}^{\infty} \frac{(tz\mathcal{S}(ruz))^i}{(1+\frac{tz+a}{r})_i} \mathcal{E}_{ir+a}(uz).
    \end{equation*}
We also define $\boldsymbol{A}^*_{a/r}$ by taking the adjoint of $\boldsymbol{A}_{a/r}$ and replacing $t$ by $-t$. 
\end{defn}
Note that these definitions come from \cite[Equation (5.19)]{J1}. We use $\boldsymbol{A}_{i/r}[k]$ to denote the $z^{k+1}$-coefficient in the expansion of $\boldsymbol{A}_{i/r}$.
\begin{thm}[\cite{J1,J}]\label{thm:orbiP1}
The full equivariant orbifold GW-potential of $\cC_{r,s}$ can be expressed as an operator expectation in the following way:
\begin{equation}\label{eqn:orbiP1}
    \tau=\left \langle e^{\sum_{k,i} x_k(i) \boldsymbol{A}_{i/r}[k]} e^{\frac{t \alpha_r}{ur}} \left(\frac{q}{t^{1/r}(-t)^{1/s}}\right)^{H} e^{\frac{-t\alpha_{-s}}{us}} e^{\sum_{k,j} x^*_k(j)\boldsymbol{A}^*_{j/s}[k]}\right\rangle
\end{equation}
\end{thm}

\section{Operator formulae for relative invariants}\label{sec:opfm}
In this section, we want to study the large $r,s$ behavior of the operator formula \eqref{eqn:orbiP1} so as to get operator formulae for relative invariants of  $(\PP^1,0)$ and $(\PP^1,0\cup \infty)$. Note that relative invariants of $(\PP^1,0)$ can be studied by setting $s=1$ and taking $r$ sufficiently large.

We will mainly consider two cases. In one case, we will study the operator formula with no large ages\footnote{Ages in the form of $\frac{r+b}{r}$, $b<0$.}. We will then recover the operator formula for equivariant relative invariants by Okounkov and Pandharipande \cite{OP3}. In another case, we will study the operator formula in the non-equivariant specialization. The limiting behavior of the operator formula will then produce operator formulae for (non-equivariant) relative invariants
with negative contact.


We will begin with an analysis of the limiting behavior of the operator formula \eqref{eqn:orbiP1}.

\subsection{Limiting behavior of the operator formula}

We only need to analyse the large $r$ behavior of \eqref{eqn:orbiP1} since the analysis for large $s$ behavior is similar.

By Theorem $\ref{thm:orbiP1}$ we know that the equivariant relative invariants of $(\PP^1,0)$ with fixed partition $\vec{\mu}_0$ of degree $d$:
\[\vec{\mu}_0=(\mu_1^+, \dots, \mu_{l_0}^+, \mu_1^-, \dots, \mu_{m_0}^-),\,\mu_i^+>0,\,\mu_j^-<0\]
can be determined by studying the following vacuum expectation\footnote{The ordering of these operators $\bA_{\bullet/r}$, $\bA_{\bullet/s}^*$ will not affect the vacuum expectation of \eqref{eqn:vacexp} by the commutator formula \cite[Lemma V.4]{J1}. But the specific ordering we chosen here will be important to deduce the operator formulae for relative invariants.}:
\begin{equation}\label{eqn:vacexp}
\begin{split}
\left\langle\prod_{i=1}^{l_0} \bA_{\mu_i^+/r}(z_i)\prod_{j=1}^{m_0}\bA_{(r+\mu_j^-)/r}(z_{l_0+j})\prod_{k=1}^{n_0}\bA_{0/r}(z_{l_0+m_0+k})e^{\frac{t \alpha_r}{ur}} \left(\frac{q}{t^{1/r}(-t)}\right)^{H} e^{\frac{-t\alpha_{-1}}{u}}\prod_{k=1}^{n_{\infty}}\bA_{0/1}^*(w_k)\right\rangle
\end{split}
\end{equation}
in the large $r$ limit. Note that we set $s=1$ in \eqref{eqn:orbiP1}. Here for simplicity we omit variables $u,t$ in the operators $\bA_{\bullet/r}$, $\bA_{0/1}^*$.

To extract the degree $d$ orbifold invariants from $\eqref{eqn:vacexp}$ we have to extract the $q^d$-term. To do this we first note that 
\[\left (\frac{q}{t^{1/r}(-t)} \right )^H = \sum_{l\geq0} \frac{q^l}{t^{l/r}(-t)^l} P_l,\]
where the operator $P_l$ is the orthogonal projection to the subspace of $\wedge^{\frac{\infty}{2}}V$ with energy $l$. Therefore we
only need to consider those operators in the expansion of 
\begin{equation}\label{eqn:s-energy}
e^{\frac{-t\alpha_{-1}}{u}}\prod_{k=1}^{n_{\infty}}\bA_{0/1}^*(w_k)
\end{equation}
with energy $d$, and those operators in 
\begin{equation}\label{eqn:r-energy}
\prod_{i=1}^{l_0}\bA_{\mu_i^+/r}(z_i)\prod_{j=1}^{m_0}\bA_{(r+\mu_j^-)/r}(z_{l_0+j})\prod_{k=1}^{n_0}\bA_{0/r}(z_{l_0+m_0+k})e^{\frac{t \alpha_r}{ur}}
\end{equation}
with energy $-d$. Those operators in \eqref{eqn:r-energy} with energy $-d$ have the following form:
\[Q_{d-kr}\left(\frac{t\alpha_{r}}{ur}\right)^k,\quad k\geq 0\]
where $Q_{d-kr}$ is a certain operator with energy $kr-d$. If $r>d$, we know that
the energy of the dual $Q^*_{d-kr}$ will be negative except when $k=0$. Note that the action of a negative energy operator on the vacuum vector $v_{\emptyset}$ is zero. So we can ignore the term $e^{\frac{t\alpha_{r}}{ur}}$ when we study the large $r$ behavior of the $q^d$-term of \eqref{eqn:vacexp}. 
\begin{rmk}
We will often meet operators $A$ on $\Lambda ^{\frac{\infty}{2}} V$ with energy of the form $kr+l$. We then call $k$ to be the \emph{$r$-energy} of $A$. For example, the $r$-energy of the operator $\mathcal{E}_{ir+a}$ is $-i$ by Remark \ref{rmk:energy}.
Similarly, if the energy of an operator $A$ is of the form $k' s+l'$, then we call $k'$ to be the \emph{$s$-energy} of $A$.
\end{rmk}

Let 
\begin{eqnarray*}
\bA_{\vec{\mu}_0,n_0} & \coloneqq & \prod_{i=1}^{l_0}\bA_{\mu_i^+/r}(z_i)\prod_{j=1}^{m_0}\bA_{(r+\mu_j^-)/r}(z_{l_0+j})\prod_{k=1}^{n_0}\bA_{0/r}(z_{l_0+m_0+k})\\
\end{eqnarray*}
From the previous discussion, we know that the study of the large $r$ limit of the $q^d$-term of \eqref{eqn:vacexp} is the same as the study of the large $r$ limit of 
\begin{equation}\label{eqn:midop}
\left\langle \bA_{\vec{\mu}_0,n_0}\left(\frac{1}{t^{d/r}(-t)^d}\right)P_d e^{\frac{-t\alpha_{-1}}{u}}\prod_{k=1}^{n_{\infty}}\bA_{0/1}^*(w_k)\right\rangle.
\end{equation}

In order to study the large $r$ limit of \eqref{eqn:midop}, we need the following structural result. First, we need to introduce our notation convention. 
For $r>|i|$, we set
\begin{equation}\label{eqn:A-new}
\boldsymbol{A}_i\coloneqq
\begin{cases}
\boldsymbol{A}_{i/r},\, i\geq 0,\\
\boldsymbol{A}_{(r+i)/r},\, i<0.
\end{cases}
\end{equation}
Let $\vec{a}=(a_1, ..., a_n)$ be an element in $\mathbb{Z}^n$. For $r>\sum_i |a_i|$, we set
\[\boldsymbol{A}_{\vec{a}}\coloneqq \prod_i \boldsymbol{A}_{a_i}(z_i,u,t).\]
For $Q=\{q_1,\cdots,q_s\}$ an ordered subset of $\{1,2,\cdots,n\}$, we will use the following convention:
\[\vec{a}_{Q}=(a_{q_1},\cdots,a_{q_s}),\quad z_{Q}=(z_{q_1},\cdots,z_{q_s})\]
and
\[|\vec{a}_Q|=\sum a_{q_i},\quad |z_{Q}|=\sum z_{q_i}. \]
Now we are ready to state our key lemma which will be established in Appendix \ref{sec:interaction}.
\begin{lem}\label{lem:basic}
Let $L_1$, $L_2$ be any two operators on $\Lambda ^{\frac{\infty}{2}} V$ with fixed energies (do not depend on $r$) and $\vec{a}\in \Z^n$. Then for sufficiently large $r$, we have the following formula:
\begin{equation}\label{eqn:keylem}
\langle L_1 \boldsymbol{A}_{\vec{a}} L_2\rangle=\sum_{P}\left \langle L_1  \left(\prod_{i=1,\cdots,l(P)}^{\longrightarrow}f_{\vec{a}_{P_i}}(z_{P_i},u,t,r)\mathcal{E}_{|\vec{a}_{P_i}|}(u|z_{P_i}|)\right)  L_2\right\rangle
\end{equation}
where $P=P_1\sqcup\cdots\sqcup P_{l(P)}$\footnote{Here elements in each $P_i$ are always ordered from small to large. Then each $P_i$ becomes an ordered subset of $\{1,2,\cdots,n\}$. So $\vec{a}_{P_i}$ and $z_{P_i}$ make sense.} runs over all the partitions of $\{1,2,\cdots,n\}$,
the operators $\mathcal{E}_{|\vec{a}_{P_i}|}$ in the product are ordered left-to-right by increasing smallest integers of $P_i$. These functions $f_{\vec{a}_{P_i}}$ do not depend on the choice of $L_1$, $L_2$ but only on the ordered set $\vec{a}_{P_i}$ (see Equation \eqref{eqn:unifunapp} in the proof). 
\end{lem}
\begin{rmk}
Note that the operators $\mathcal{E}_{|\vec{a}_{P_i}|}$ do not commute, so it is important to indicate their ordering.
If we assume that the smallest integer in $P_i$ is always smaller than that of $P_j$, $\forall i<j$, then we could also denote the ordered product as
\[\prod_{i=1}^{l(P)}f_{\vec{a}_{P_i}}(z_{P_i},u,t,r)\mathcal{E}_{|\vec{a}_{P_i}|}(u|z_{P_i}|).\]
\end{rmk}

Recall that we are interested in relative invariants. Therefore we consider
\[
r^{m_0} \left\langle \bA_{\vec{\mu}_0,n_0}\left(\frac{1}{t^{d/r}(-t)^d}\right)P_d e^{\frac{-t\alpha_{-1}}{u}}\prod_{k=1}^{n_{\infty}}\bA_{0/1}^*(w_k)\right\rangle.
\]
for $r$ sufficiently large. Note that $m_0$ is the number of negative integers in $\vec{\mu}_0$. Since $P_d e^{\frac{-t\alpha_{-1}}{u}}\prod_{k=1}^{n_{\infty}}\bA_{0/1}^*(w_k)$ is an operator with fixed energy $d$, we can apply Lemma~\ref{lem:basic} and get for sufficiently large $r$
\[
\frac{r^{m_0}}{t^{d/r} (-t)^{d}}\left\langle \sum_{P} \prod_{i=1}^{l(P)}f_{\vec{a}_{P_i}}(z_{P_i},u,t,r)\mathcal{E}_{|\vec{a}_{P_i}|}(u|z_{P_i}|) e^{\frac{-t\alpha_{-1}}{u}}\prod_{k=1}^{n_{\infty}}\bA_{0/1}^*(w_k)\right\rangle
\]
where $\vec{a}= (\vec{\mu}_0, \underbrace{0, ..., 0}_{n_0})$. Note that 
$$\sum_{P} \prod_{i=1}^{l(P)}f_{\vec{a}_{P_i}}(z_{P_i},u,t,r)\mathcal{E}_{|\vec{a}_{P_i}|}(u|z_{P_i}|)$$
is an operator with energy $-d$. This means only the energy $d$ terms on the right side of $P_d$ actually contribute to the operator expectation. So we can omit the projection operator $P_d$. 

For each partition $P$, if we could derive an explicit formula for the  $r^0$-term of 
\begin{equation}\label{eqn:f_cons}
r^{m_0}\prod_{i=1}^{l(P)}f_{\vec{a}_{P_i}}(z_{P_i},u,t,r),
\end{equation}
then we would get an explicit formula for the equivariant relative GW-theory of $(\mathbb{P}^1, 0)$ with negative contact orders. But at the moment we are only able to determine the $r^0$-term of \eqref{eqn:f_cons} for some special $\vec{a}$.

\subsection{Operator formulae for equivariant relative invariants}\label{sec:opeqv}
The first special case where the $r^0$-term of \eqref{eqn:f_cons} can be explicitly computed is  $\vec{a}=(0, ..., 0)$. The determination of the $r^0$-terms of \eqref{eqn:f_cons} for all such $\vec{a}$ will be the key step to establish the operator formulae for the equivariant relative theory of the cap and tube (without negative contact). We start from the cap case.
\subsubsection{Cap case}\label{section:equicap}
We mainly consider equivariant relative invariants of $(\PP^1,0)$. The case of $(\PP^1,\infty)$ can be similarly determined which provides a new way to the operator formula in \cite[Proposition 3.2]{OP3} (see Remark \ref{rmk:OPfm}).

So we are interested in the following generating series of equivariant GW-invariants of $(\PP^1,0)$:
\begin{equation}\label{eqn:eqvcap}
G(\vec{\mu}|z_1,\cdots,z_n,w_1,\cdots,w_m)\coloneqq\sum_{k_i,l_i}\prod_{i=1}^n z_i^{k_i+1}\prod_{j=1}^m w_j^{l_j+1}\langle\vec{\mu}|\prod_{i=  }^n\tau_{k_i}(\boldsymbol{0})\prod_{j=1}^m\tau_{l_j}(\boldsymbol{\infty})\rangle^{\bullet,\mathbb{C}^*}
\end{equation}
where $\vec{\mu}=\{\mu_1,\cdots,\mu_\rho\}$ is an ordered partition of $d$ such that each $\mu_i>0$, $\boldsymbol{0}$ and $\boldsymbol{\infty}$ are two natural equivariant classes of $\PP^1$, and the superscripts $\bullet$, $\mathbb{C}^*$ indicate that we consider possibly disconnected equivariant relative invariants of $(\PP^1,0)$. Here as in \cite{OP3}, we omit the variable for genus in the definition of the generating series $G$ since it is redundant. We also need to add unstable contributions in \eqref{eqn:eqvcap} to achieve uniformity (see \cite[Section 2.3.4]{OP3}).

To compute \eqref{eqn:eqvcap}, we need to consider the large $r$ behavior of
\begin{equation}\label{eqn:expfm-I}
\left\langle \bA_{\vec{\mu},n}\left(\frac{1}{t^{d/r}(-t)^d}\right)P_d e^{\frac{-t\alpha_{-1}}{u}}\prod_{j=1}^{m}\bA_{0/1}^*(w_j)\right\rangle
\end{equation}
where 
\[\bA_{\vec{\mu},n}=\prod_{i=1}^{\rho}\bA_{\mu_i/r}(x_i)\prod_{i=1}^{n}\bA_{0/r}(z_i).\]
Note that we do not have any psi-class insertions along the relative markings. So we only need to know the coefficient of $x_i$:
\[\bA_{\mu_i/r}[0]= \left[\frac{t^{\frac{\mu_i}{r}}}{u}\frac{x_i}{tx_i+\mu_i} \mathcal{S}(rux_i)^{\frac{tx_i+\mu_i}{r}} \sum_{i=-\infty}^{\infty} \frac{(tx_i\mathcal{S}(rux_i))^i}{(1+\frac{tx_i+\mu_i}{r})_i} \mathcal{E}_{ir+\mu_i}(ux_i)\right]_{x_i}\]
where $[\cdot]_{x_i}$ means that we take the coefficient of $x_i$. It is easy to see from the expression that only those summands with non-positive index $i$ will contribute to $\bA_{\mu_i/r}[0]$. But if $i<0$, the dual operator $\mathcal{E}_{ir+\mu_i}^*$ has negative energy for $r>\mu_i$, then $\mathcal{E}_{ir+\mu_i}^*v_{\emptyset}=0$. So only the term 
\[\left[\frac{t^{\frac{\mu_i}{r}}}{u}\frac{x_i}{tx_i+\mu_i} \mathcal{S}(rux_i)^{\frac{tx_i+\mu_i}{r}}\mathcal{E}_{\mu_i}(ux_i)\right]_{x_i}=\frac{t^{\frac{\mu_i}{r}}}{u\mu_i}\alpha_{\mu_i}\]
will be needed in the large $r$ behavior of \eqref{eqn:expfm-I}. Recall that $\alpha_{\mu_i}=\mathcal{E}_{\mu_i}(0)$. So if we set 
\begin{eqnarray*}\label{eqn:l-ops}
L_1 & = & \prod_{i=1}^{\rho}\frac{t^{\frac{\mu_i}{r}}}{u\mu_i}\alpha_{\mu_i},\\
L_2 & = & \text{energy $d$ part of the operator }e^{\frac{-t\alpha_{-1}}{u}}\prod_{j=1}^{m}\bA_{0/1}^*(w_j),
\end{eqnarray*}
then \eqref{eqn:expfm-I} can be determined from
\begin{equation}\label{eqn:middleterm}
\left(\frac{1}{t^{1/r}(-t)}\right)^{d}\left\langle L_1 \prod_{i=1}^{n}\bA_{0/r}(z_i) L_2\right\rangle.
\end{equation}
We then apply Lemma \ref{lem:basic} to \eqref{eqn:middleterm} which gives us
\begin{equation}\label{eqn:3rdterm}
\left(\frac{1}{t^{1/r}(-t)}\right)^{d}\sum_{P}\left \langle L_1  \left(\prod_{i=1}^{l(P)}f_{\vec{a}_{P_i}}(z_{P_i},u,t,r)\mathcal{E}_0(u|z_{P_i}|)\right)  L_2\right\rangle
\end{equation}
for $r$ sufficiently large. 
Note that in this case, all the $\vec{a}_{P_i}$ are in the form of 
\[\underbrace{(0,0,\cdots,0)}_{l_i}\]
where $l_i$ denotes the number of elements in $P_i$. Since the operators 
$\mathcal{E}_0(u|z_{P_i}|)$ commute with each other, the ordering is not important in this case. Given such an $\vec{a}_{P_i}$, the following lemma shows that $f_{\vec{a}_{P_i}}$ can be determined from the Hodge integrals on the moduli space of curves: 
\begin{lem}\label{lem:f0-hodge}
Let $P_i=\{p_1,p_2,\cdots,p_{l_i}\}$ and $\vec{a}_{P_i}=(0,0,\cdots,0)$. We have
\[f_{\vec{a}_{P_i}}=\varsigma (ru|z_{P_i}|)\frac{r^{2l_i}}{t^{l_i}}\sum_{g\geq 0}\left(\frac{r^2u}{t}\right)^{2g-2}H_g^{\circ}\left(\frac{tz_{p_1}}{r},\dots,\frac{tz_{p_{l_i}}}{r}\right)\]
where
\[H_g^{\circ}(a_1,\cdots,a_{l_i})=\prod_{j}a_j\int_{[\bM_{g,l_i}]}\frac{1-\lambda_1+\lambda_2-\dots\pm\lambda_g}{\prod_j(1-a_j\psi_j)}\]
if $2g-2+l_i>0$. Otherwise, we set 
\begin{equation}\label{eqn:unstable}
H_0^{\circ}(z_1,z_2)=\frac{z_1z_2}{z_1+z_2},\quad H_0^{\circ}(z_1)=\frac{1}{z_1}.
\end{equation}
Note that $\varsigma(z)$ is given by \eqref{eqn:vsig}.
\end{lem}
\begin{proof}
By \eqref{eqn:unifunapp}, we know that
\[\langle f_{\vec{a}_{P_i}}\mathcal{E}_0(u|z_{P_i}|)\rangle=\langle \boldsymbol{A}_{0/r}(z_{p_1})\dots\boldsymbol{A}_{0/r}(z_{p_{l_i}})  \rangle^{\circ}\]
where $\langle\dots\rangle^{\circ}$ denotes the connected part of the vacuum expectation (see Section 3.3 in \cite{OP1}). The trick here is to compare $\langle \boldsymbol{A}_{0/r}(z_{p_1})\dots\boldsymbol{A}_{0/r}(z_{p_{l_i}})  \rangle^{\circ}$ with 
\[\left\langle\prod_{j=1}^{l_i}\mathcal{A}\left(\frac{tz_{p_j}}{r},ruz_{p_j}\right)\right\rangle^{\circ}\]
where
\[\mathcal{A}(a,b)=\mathcal{S}(b)^a\sum_{k\in\Z}\frac{(\varsigma(b))^k}{(a+1)_k}\mathcal{E}_k(b).\]
Recall that functions $\mathcal{S}(z)$ and $\varsigma(z)$ are given by \eqref{eqn:S-fun} and \eqref{eqn:vsig} respectively.
By comparing operators $\mathcal{A}$ and $\boldsymbol{A}_{0/r}$,
it is easy to see that
\[u^{-l_i}\left\langle\prod_{j=1}^{l_i}\mathcal{A}\left(\frac{tz_{p_j}}{r},ruz_{p_j}\right)\right\rangle^{\circ}=\langle f_{\vec{a}_{P_i}}\mathcal{E}_0(ru(\sum z_i))\rangle\]
The RHS is further equal to $\frac{1}{\varsigma (ru|z_{P_i}|)}\cdot f_{\vec{a}_{P_i}}$.
Thus it gives a way to compute $f_{\vec{\alpha}}$ once the LHS is known.
By \cite[Theorem 2]{OP2}, the LHS equals to 
\[\frac{r^{2l_i}}{t^{l_i}}\sum_{g\geq 0}\left(\frac{r^2u}{t}\right)^{2g-2}H_g^{\circ}\left(\frac{tz_{p_1}}{r},\dots,\frac{tz_{p_{l_i}}}{r}\right)\]
The lemma then follows.
\end{proof}

Using Lemma \ref{lem:f0-hodge}, we can then
extract the $r^0$-coefficient of \eqref{eqn:3rdterm} and get the operator formula for $G(\vec{\mu}|z_1,\cdots,z_n,w_1,\cdots,w_n)$. We need the following notation convention from \cite{OP3}. Let 
\[\begin{aligned}
T(x_1,\cdots,x_k)&=x_1\cdots x_k\left(\sum_{i=1}^k x_i\right)^{k-2},\\
\mathcal{E}_0(x_1,\cdots,x_k)&=T(x_1,\cdots,x_k)\mathcal{E}_0(x_1+\cdots+x_k).
\end{aligned}\]
We further set 
\[E(x_1,\cdots,x_n,s)=\sum_{\pi}s^{n-l(\pi)}\prod_{k=1}^{l(\pi)}\mathcal{E}_0(x_{\pi_k})\]
where $\pi=\pi_1\cup\pi_2\cdots\cup \pi_{l(\pi)}$ takes over all the partitions of $\{1,2,\cdots,n\}$. 

Let
\[\bA(z)=\mathcal{S}(z)^{tz}\sum_{k\in\Z}\frac{(\varsigma(z))^k}{(tz+1)_k}\mathcal{E}_k(z)\]
and $\bA^*(z)$ be the adjoint of $\bA(z)$ with $t$ replaced by $-t$. We have
\begin{thm}\label{thm:opfm_cap0}
\[G(\vec{\mu}|z_1,\cdots,z_n,w_1,\cdots,w_n)=\left\langle  \prod_{i=1}^{\rho}\frac{\alpha_{\mu_i}}{\mu_i}E(z_1,\cdots,z_n,t)e^{\alpha_{-1}}\prod_{j=1}^m \bA^*(w_j)\right\rangle.\]
\end{thm}
\begin{proof}
We only need to take the $r^0$-coefficient of \eqref{eqn:3rdterm} and set $u=1$. First note that the factor $\left(\frac{1}{t^{1/r}(-t)}\right)^{d}$ will cancel with the corresponding powers of $t$ in $L_1$ and $L_2$. After cancelling with the $t$-powers and setting $u=1$, $L_1$ becomes $\prod_{i=1}^{\rho}\frac{\alpha_{\mu_i}}{\mu_i}$ and $L_2$ becomes\footnote{Here in \eqref{eqn:3rdterm}, $L_1  \left(\prod_{i=1}^{l(P)}f_{\vec{a}_{P_i}}(z_{P_i},u,t,r)\mathcal{E}_0(u|z_{P_i}|)\right)$ has energy $-d$, so we could directly take $L_2$ to be $e^{\frac{-t\alpha_{-1}}{u}}\prod_{j=1}^{m}\bA_{0/1}^*(w_j)$ in \eqref{eqn:3rdterm}.} $e^{\alpha_{-1}}\prod_{j=1}^m \bA^*(w_j)$. By comparing the operators $E(z_1,\cdots,z_n,t)$ and \[\sum_{P}\prod_{i=1}^{l(P)}f_{\vec{a}_{P_i}}(z_{P_i},u,t,r)\mathcal{E}_0(u|z_{P_i}|),\]
we only need to show that after setting $u=1$, the $r^0$-coefficient of $f_{\vec{a}_{P_i}}$ equals to \[t^{l_i-1}\left(\prod_{p_j\in P_i}z_{p_j}\right)(|z_{P_i}|)^{l_i-2}.\]
Note that here $\vec{a}_{P_i}$ has the form of $(0,\cdots,0)$. So by Lemma \ref{lem:f0-hodge}, we know that 
\[f_{\vec{a}_{P_i}}=\sum_{g\geq 0}\varsigma (ru|z_{P_i}|)\frac{r^{2l_i}}{t^{l_i}}\left(\frac{r^2u}{t}\right)^{2g-2}H_g^{\circ}\left(\frac{tz_{p_1}}{r},\dots,\frac{tz_{p_{l_i}}}{r}\right).\]
The key point is that only the genus zero summand will contribute 
to the $r^0$-coefficient of $f_{\vec{a}_{P_i}}$. The reason is that for each summand
\[\varsigma (ru|z_{P_i}|)\frac{r^{2l_i}}{t^{l_i}}\left(\frac{r^2u}{t}\right)^{2g-2}H_g^{\circ}\left(\frac{tz_{p_1}}{r},\dots,\frac{tz_{p_{l_i}}}{r}\right),\]
$\varsigma (ru|z_{P_i}|)\frac{r^{2l_i}}{t^{l_i}}\left(\frac{r^2u}{t}\right)^{2g-2}$ will contribute at least $1+2l_i+4g-4$ to the powers of $r$ and $H_g^{\circ}$ will contribute at least $-l_i-(3g-3+l_i)$ to the powers of $r$. So totally, it will contribute at least $g$ to the powers of $r$. Then for $g>0$, such a summand will not contribute to the $r^0$-coefficient. When $g=0$, it is easy to compute that the contribution to the $r^0$-coefficient is exactly
\[t^{l_i-1}\left(\prod_{p_j\in P_i}z_{p_j}\right)(|z_{P_i}|)^{l_i-2}.\]
Here we have used the well-known equality 
\[\int_{[\bM_{0,l_i}]}\frac{1}{\prod_j(1-a_j\psi_j)}=\left(\sum_{j}a_j\right)^{l_i-3}\]
and the unstable convention \eqref{eqn:unstable}.
\end{proof}

\begin{rmk}\label{rmk:OPfm}
We can similarly show that the generating series of equivariant relative invariants of $(\PP^1,\infty)$:
\[G(z_1,\cdots,z_n,w_1,\cdots,w_n|\vec{\nu})\coloneqq\sum_{k_i,l_i}\prod_{i=1}^n z_i^{k_i+1}\prod_{j=1}^m w_j^{l_j+1}\langle\prod\tau_{k_i}(\boldsymbol{0})\prod\tau_{l_j}(\boldsymbol{\infty})\left |\vec{\nu}\right.\rangle^{\bullet,\mathbb{C}^*}\]
has the following operator formula:
\[\left\langle \prod_{i} \bA(z_i)e^{\alpha_1}E(w_1,\cdots,w_m,-t)\prod_{j}\frac{\alpha_{v_j}}{v_j}  \right\rangle.\]
Up to an automorphism factor, it coincides with the operator formula in \cite[Proposition 3.2]{OP3}. The extra automorphism factor arises because in \cite{OP3}, relative markings are unordered.
\end{rmk}

\begin{rmk}\label{rmk:r-degbd}
We may further write the function $f_{\vec{a}_{P_i}}$ in Lemma \ref{lem:f0-hodge} as 
\[\varsigma (ru|z_{P_i}|)\frac{r^{2l_i}}{t^{l_i}}\sum_{g\geq 0}\left(\frac{r^2u}{t}\right)^{2g-2}\left(H_{g,1}^{\circ}\left(\frac{tz_{p_1}}{r},\dots,\frac{tz_{p_{l_i}}}{r}\right)+H_{g,2}^{\circ}\left(\frac{tz_{p_1}}{r},\dots,\frac{tz_{p_{l_i}}}{r}\right)\right)\]
where 
\[H_{g,1}^{\circ}(a_1,\cdots,a_{l_i})=\prod_{j}a_j\int_{[\bM_{g,l_i}]}\frac{(-1)^g\lambda_g}{\prod_j(1-a_j\psi_j)},\quad H_{g,2}^{\circ}=H_{g}^{\circ}-H_{g,1}^{\circ}\]
if $2g-2+l_i>0$. Otherwise we set 
\[H_{0,1}^{\circ}(z_1,z_2)=H_0^{\circ}(z_1,z_2),\, H_{0,2}^{\circ}(z_1,z_2)=0;\, H_{0,1}^{\circ}(z_1)=H_{0}(z_1),\,H_{0,2}^{\circ}(z_1)=0.\]
The integrals with insertion of $\lambda_g$ were explicitly calculated in \cite{FP, FP2}, we have
\[\varsigma (ru|z_{P_i}|)\frac{r^{2l_i}}{t^{l_i}}\sum_{g\geq 0}\left(\frac{r^2u}{t}\right)^{2g-2}H_{g,1}^{\circ}\left(\frac{tz_{p_1}}{r},\dots,\frac{tz_{p_{l_i}}}{r}\right)=\frac{t^{l_i-1}}{u}z_1\cdots z_{p_{l_i}}|z_{P_i}|^{l_i-2}.\]
So the function $f_{\vec{a}_{P_i}}$ in Lemma \ref{lem:f0-hodge} can be further decomposed as
\[\frac{t^{l_i-1}}{u}z_1\cdots z_{p_{l_i}}|z_{P_i}|^{l_i-2}+\varsigma (ru|z_{P_i}|)\frac{r^{2l_i}}{t^{l_i}}\sum_{g\geq 1}\left(\frac{r^2u}{t}\right)^{2g-2}H_{g,2}^{\circ}\left(\frac{tz_{p_1}}{r},\dots,\frac{tz_{p_{l_i}}}{r}\right).\]
Here the summation starts from $g=1$ because $H_{g=0,2}^{\circ}$ always equals to zero. The same argument as in the proof of Theorem \ref{thm:opfm_cap0} shows that the power of $r$ for each term in the expansion of
\[f_{\vec{a}_{P_i},2}\coloneqq\varsigma (ru|z_{P_i}|)\frac{r^{2l_i}}{t^{l_i}}\sum_{g\geq 1}\left(\frac{r^2u}{t}\right)^{2g-2}H_{g,2}^{\circ}\left(\frac{tz_{p_1}}{r},\dots,\frac{tz_{p_{l_i}}}{r}\right)\]
is positive. Note that for a fixed $g\geq 1$, the $r$-degree of $H_{g,2}^{\circ}$ ranges from $-(3g-3+2l_i)$ to $-(2g-2+2l_i)$, this will imply that
the power of $r$ is no bigger than that of $u$ for any term in the expansion of $f_{\vec{a}_{P_i},2}$. In
\[f_{\vec{a}_{P_i},1}\coloneqq \frac{t^{l_i-1}}{u}z_1\cdots z_{p_{l_i}}|z_{P_i}|^{l_i-2},\]
the power of $r$ is zero which is one bigger than that of $u$. 
So for each term with positive power of $r$ in the expansion of $$\prod_{i=1}^{l(P)}f_{\vec{a}_{P_i}}=\prod_{i=1}^{l(P)}(f_{\vec{a}_{P_i},1}+f_{\vec{a}_{P_i},2}),$$
the power of $r$ is at most $l(P)-1$ greater than that of $u$. But if we compute the connected part of the expectation:
\begin{equation}\label{eqn:connectedpart}
\left(\frac{1}{t^{1/r}(-t)}\right)^{d}\left \langle L_1  \left(\prod_{i=1}^{l(P)}\mathcal{E}_0(u|z_{P_i}|)\right)  L_2\right\rangle^{\circ}
\end{equation}
and compare with the number of functions of the form 
$$\varsigma(u(a_1z_1+\cdots+a_nz_n+b_1w_1+\cdots b_mw_m)),\quad a_1,\cdots,a_n,b_1,\cdots,b_m\in \Z,$$
created from the commutation relation \eqref{eq:commrel} and the powers of $u$ in the operators $L_1,L_2$, we conclude that the power of $u$ for each term in the expansion of \eqref{eqn:connectedpart} is at least $l(P)-2$ and no power of $r$ appears. 

So all in all, for each term with positive power of $r$ in the expansion of the connected part of the expectation \eqref{eqn:3rdterm},
the power of $r$ is at most one greater than that of $u$. Since the connected part of the expectation \eqref{eqn:3rdterm} computes the connected orbifold invariants of $\mathcal{C}_{r,1}$, this implies that the connected genus $g$ orbifold invariants of $C_{r,1}$, as a polynomial of $r$ for sufficiently large $r$, have $r$-degree bounded by $\max\{0,2g-1\}$ (because $u$ has power $2g-2$). The same $r$-degree bound was shown in \cite{TY2} for arbitrary root stack $X_{D,r}$. Our explicit calculation provides a new perspective on the $r$-degree bound in the case of $\mathcal{C}_{r,1}$.

\end{rmk}

\subsubsection{Tube case} In this case, we consider the generating series of equivariant relative invariants of $(\PP^1,0\cup\infty)$:
\[G(\vec{\mu}|z_1,\cdots,z_n,w_1,\cdots,w_n|\vec{\nu})\coloneqq\sum_{k_i,l_i}\prod_{i=1}^n z_i^{k_i+1}\prod_{j=1}^m w_j^{l_j+1}\langle\vec{\mu}|\prod\tau_{k_i}(\boldsymbol{0})\prod\tau_{l_j}(\boldsymbol{\infty})\left |\vec{\nu}\right.\rangle^{\bullet,\mathbb{C}^*}\]
In order to compute it, we need to take both $r$ and $s$ limits of the operator formula \eqref{eqn:orbiP1}. But the analysis is parallel to the cap case. So we omit the details and just write down the final result:
\begin{thm}\label{thm:eqtubefm}
\[G(\vec{\mu}|z_1,\cdots,z_n,w_1,\cdots,w_n|\vec{\nu})=\left\langle  \prod_{i}\frac{\alpha_{\mu_i}}{\mu_i}E(z_1,\cdots,z_n,t)E(w_1,\cdots,w_m,-t)\prod_{j}\frac{\alpha_{v_j}}{v_j}  \right\rangle.\]
\end{thm}

\subsection{Operator formulae for relative invariants with negative contact}\label{sec:noneqopfm}
In this subsection, we consider relative invariants in the non-equivariant case. By analysing the large $r,s$ behavior of the operator formula \eqref{thm:orbiP1} in the non-equivariant specialization, we can determine the operator formulae for relative invariants of the cap and tube with negative contact. The key step will be the explicit calculation of a certain coefficient of the $r^0$-term of \eqref{eqn:f_cons} when all the elements of $\vec{a}$ are negative. As before, we start from the cap case.

\subsubsection{Cap case}
We are interested in computing the following relative invariant of $(\PP^1,0)$:
\begin{equation}\label{eqn:noneq-cap}
\left\langle \vec{\mu}_0\left|\prod_{i=1}^n\tau_{k_i}(\omega)\right.\right\rangle^{\bullet}
\end{equation}
where $\vec{\mu}_0=(a_1, \dots, a_{l_0}, b_1, \dots, b_{m_0})$ is an ordered partition of the degree $d$ such that $a_i>0$, $b_j<0$, and $\omega$ is the point class of $\PP^1$. It is equivalent to consider the following equivariant invariant:
\begin{equation}\label{eqn:eq-cap}
\left\langle \vec{\mu}_0\left|\prod_{i=1}^n\tau_{k_i}(\boldsymbol{0})\right.\right\rangle^{\bullet,\C^*}.
\end{equation}
So we need to analyse the large $r$ behavior of the following vacuum expectation in the non-equivariant specialization:
\begin{equation}\label{eqn:op-noneqcap}
r^{m_0}\left\langle \bA_{\vec{\mu}_0,n}\left(\frac{1}{t^{d/r}(-t)^d}\right)P_d e^{\frac{-t\alpha_{-1}}{u}}\right\rangle
\end{equation}
where 
\[\bA_{\vec{\mu}_0,n}= \prod_{i=1}^{l_0}\bA_{a_i/r}(x_i)\prod_{j=1}^{m_0}\bA_{(r+b_j)/r}(y_j)\prod_{k=1}^{n}\bA_{0/r}(z_k).\]
Since we have no psi-class insertions along the relative markings, a similar analysis as in Section \ref{section:equicap}
shows that we could replace the product $\prod_{i=1}^{l_0}\bA_{a_i/r}(x_i)$ by $\prod_{i=1}^{l_0}\frac{t^{a_i/r}}{ua_i}\alpha_{a_i}$. After plugging into \eqref{eqn:op-noneqcap} and cancelling certain powers of $t$, it becomes
\begin{equation}\label{eqn:midterm-noneqcap}
\frac{r^{m_0}}{t^{(\sum b_j)/r}u^{l_0+d}}\left\langle \prod_{i=1}^{l_0}\frac{\alpha_{a_i}}{a_i}\prod_{j=1}^{m_0}\bA_{(r+b_j)/r}(y_j)\prod_{k=1}^{n}\bA_{0/r}(z_k)\frac{(\alpha_{-1})^d}{d!} \right\rangle.
\end{equation}
We observe that the $t$-power in the numerator of each summand of
\begin{eqnarray*}
\bA_{(r+b_j)/r}(y_j) & = & \frac{t^{\frac{b_j}{r}}}{u}\frac{y_j}{ty_j+r+b_j} \mathcal{S}(ruy_j)^{\frac{ty_j+r+b_j}{r}} \sum_{i=-\infty}^{\infty} \frac{t^{i+1}(y_j\mathcal{S}(ruy_j))^i}{(1+\frac{ty_j+r+b_j}{r})_i} \mathcal{E}_{(i+1)r+b_j}(uy_j),\\
\bA_{0/r}(z_k)       & = & \frac{1}{u} \mathcal{S}(ruz_k)^{\frac{tz_k}{r}} \sum_{i=-\infty}^{\infty} \frac{t^i(z_k\mathcal{S}(ruz_k))^i}{(1+\frac{tz_k}{r})_i} \mathcal{E}_{ir}(uz_k),
\end{eqnarray*}
is the opposite to the $r$-energy of the corresponding operator,
and only $r$-energy zero operators in the expansion of the product 
$$\prod_{j=1}^{m_0}\bA_{(r+b_j)/r}(y_j)\prod_{k=1}^{n}\bA_{0/r}(z_k)$$
will contribute to the vacuum expectation  \eqref{eqn:midterm-noneqcap}. So we can omit those powers of $t$ and define
\begin{eqnarray*}
\bA_{(r+b_j)/r}'(y_j) & = & \frac{y_j}{ty_j+r+b_j} \mathcal{S}(ruy_j)^{\frac{ty_j+r+b_j}{r}} \sum_{i=-\infty}^{\infty} \frac{(y_j\mathcal{S}(ruy_j))^i}{(1+\frac{ty_j+r+b_j}{r})_i} \mathcal{E}_{(i+1)r+b_j}(uy_j)\\
\bA_{0/r}'(z_k)       & = & \mathcal{S}(ruz_k)^{\frac{tz_k}{r}} \sum_{i=-\infty}^{\infty} \frac{(z_k\mathcal{S}(ruz_k))^i}{(1+\frac{tz_k}{r})_i} \mathcal{E}_{ir}(uz_k).
\end{eqnarray*}
Now \eqref{eqn:midterm-noneqcap} is equal to
\begin{equation}\label{eqn:3rdterm-noneqcap}
\frac{r^{m_0}}{u^{l_0+d+m_0+n}}\left\langle \prod_{i=1}^{l_0}\frac{\alpha_{a_i}}{a_i}\prod_{j=1}^{m_0}\bA'_{(r+b_j)/r}(y_j)\prod_{k=1}^{n}\bA'_{0/r}(z_k)\frac{(\alpha_{-1})^d}{d!} \right\rangle.
\end{equation}
The good part of the above formula is that we can take the non-equivariant limit, i.e., set $t=0$. By the definition of the Pochhammer symbol \eqref{eqn:Psymb},
we know that only those summands in $\bA_{0/r}'(z_k)$  with index $i\geq 0$ will survive when we set $t=0$. So the $r$-energy of each operator in the expansion of $\prod_{k=1}^{n}\bA'_{0/r}(z_k,t=0)$ is always non-positive. Let $R$ be an operator in the expansion of $\prod_{k=1}^{n}\bA'_{0/r}(z_k,t=0)$ with negative $r$-energy. Then 
$$\left(R\frac{(\alpha_{-1})^d}{d!}\right)v_{\emptyset}=0$$
when $r$ becomes sufficiently large. So we only need to consider the zero $r$-energy operator in $\prod_{k=1}^{n}\bA'_{0/r}(z_k,t=0)$ which is just
\[\prod_{k=1}^{n}\mathcal{E}_0(uz_k).\]
So the non-equivariant limit of \eqref{eqn:3rdterm-noneqcap} is reduced to the following vacuum expectation:
\begin{equation}\label{eqn:4thterm-noneqcap}
\frac{r^{m_0}}{u^{l_0+d+m_0+n}}\left\langle \prod_{i=1}^{l_0}\frac{\alpha_{a_i}}{a_i}\prod_{j=1}^{m_0}\bA'_{(r+b_j)/r}(y_j,t=0)\prod_{k=1}^{n}\mathcal{E}_0(uz_k)\frac{(\alpha_{-1})^d}{d!} \right\rangle.
\end{equation}
Lemma \ref{lem:basic} still holds by replacing $\bA_{\vec{a}}$ with $\prod_{j=1}^{m_0}\bA'_{(r+b_j)/r}(y_j,t=0)$. So
\eqref{eqn:4thterm-noneqcap} can be written as
\begin{equation}\label{eqn:midfm-noneqcap}
\frac{r^{m_0}}{u^{l_0+d+m_0+n}}\sum_P\left\langle \prod_{i=1}^{l_0}\frac{\alpha_{a_i}}{a_i}\prod_{j=1}^{l(P)}f_{\vec{b}_{P_j}}\mathcal{E}_{|\vec{b}_{P_j}|}(u|y_{P_j}|)\prod_{k=1}^{n}\mathcal{E}_0(uz_k)\frac{(\alpha_{-1})^d}{d!} \right\rangle
\end{equation}
where $\vec{b}=(b_1,\cdots,b_{m_0})$, and $P=P_1\sqcup\cdots \sqcup P_{l(P)}$ runs over all the partitions of $\{1,2,\cdots,m_0\}$. Here we always assume that the smallest integer in $P_i$ is smaller than that of $P_j$, $\forall i<j$. 

Our aim is to compute \eqref{eqn:noneq-cap} which has no psi-class insertions along the relative markings. So we only need to determine the $(y_{P_j})!$-coefficient of $f_{\vec{b}_{P_j}}\mathcal{E}_{|\vec{b}_{P_j}|}(u|y_{P_j}|)$ for each $P_j$.
Here we have used the convention that 
\[(y_{P_j})!=\prod_{i\in P_j}y_i.\]

We have the following structural result for the $(y_{P_j})!$-coefficient:
\begin{lem}\label{lem:formalneg}
The  $(y_{P_j})!$-coefficient of  $f_{\vec{b}_{P_j}}\mathcal{E}_{|\vec{b}_{P_j}|}(u|y_{P_j}|)$ has the following form:
\[C_{P_j}(r)u^{l(P_j)}\mathcal{E}_{|\vec{b}_{P_j}|}[0]\]
where $C_{P_j}(r)$ is a rational function of $r$ which depends on $P_j$, $l(P_j)$ denotes number of elements in $P_j$, and $\mathcal{E}_{|\vec{b}_{P_j}|}[0]$ denotes the $z$-coefficient of $\mathcal{E}_{|\vec{b}_{P_j}|}(z)$.
\end{lem}
\begin{proof}
The proof relies on some facts from Appendix \ref{sec:interaction}.

Let $R_{m_i,b_i}(y_i)$ be the summand in $\bA_{(r+b_i)/r}'(y_i,t=0)$ with $r$-energy $-m_i$. More precisely, we have
\[R_{m_i,b_i}(y_i)=\frac{1}{r+b_i}  \frac{(\mathcal{S}(ruy_i))^{m_i+\frac{b_i}{r}}}{(1+\frac{r+b_i}{r})_{m_i-1}}y_i^{m_i} \mathcal{E}_{m_ir+b_i}(uy_i).\]
By Equation \eqref{eqn:unifunapp} and Definition \ref{def:G-fun}, we know that $f_{\vec{b}_{P_j}}\mathcal{E}_{|\vec{b}_{P_j}|}(u|y_{P_j}|)$ can be expressed as
\[\sum_{|\vec{m}_{P_j}|=0}\sum_{\mathcal{J}\in  \mathcal{K}(R_{\vec{m}_{P_j},\vec{b}_{P_j}})\atop \mathcal{J}\text{ is connected}}L(\mathcal{J})(R_{\vec{m}_{P_j},\vec{b}_{P_j}})\]
where $\vec{m}_{P_j}=(m_i)_{i\in P_j}$, $R_{\vec{m}_{P_j},\vec{b}_{P_j}}=(R_{m_i,b_i})_{i\in P_j}$, the interaction diagrams $\mathcal{J}$ indicate how the operators in $R_{\vec{m}_{P_j},\vec{b}_{P_j}}$ interact with each other which yields the operators $L(\mathcal{J})(R_{\vec{m}_{P_j},\vec{b}_{P_j}})$ (see Definition \ref{def:nestedLie}). So we only take the $(y_{P_j})!$-coefficient for each $L(\mathcal{J})(R_{\vec{m}_{P_j},\vec{b}_{P_j}})$ which can be written as
\begin{equation}\label{eqn:fmexmidterm}
f_{\mathcal{J}}(y_{P_j},u,r)\mathcal{E}_{|\vec{b}_{P_j}|}(u|y_{P_j}|)
\end{equation}
by the commutation relation \eqref{eq:commrel} and the fact that $\mathcal{J}$ is connected. 

We define the $y$-degree of a monomial in the variables $y_{P_j}$ to be the total sum of degrees of different $y_i\in y_{P_j}$. For example, $(y_{P_j})!$ has $y$-degree $l(P_j)$. We claim that each monomial in the expansion of $f_{\mathcal{J}}(y_{P_j},u,r)$ has $y$-degree at least $l(P_j)-1$ and there is no monomial with $y$-degree $l(P_j)$.

The reason is that we can express $f_{\mathcal{J}}$ as 
\[\left(\prod_{i\in P_j}\frac{1}{r+b_i}  \frac{(\mathcal{S}(ruy_i))^{m_i+\frac{b_i}{r}}}{(1+\frac{r+b_i}{r})_{m_i-1}}y_i^{m_i}\right)\varsigma_{\mathcal{J}}\]
where $\varsigma_{\mathcal{J}}$ is a product of $l(P_j)-1$ $\varsigma$ functions (see \eqref{eqn:vsig} for the definition of $\varsigma$ function) which arises from the commutation relation \eqref{eq:commrel}. Each $\varsigma$ function appearing in the product has the following form:
\[\varsigma \left(u\left (\sum_{i\in P_j} c_i(r)y_i\right)\right)\]
where $c_i(r)$ are certain linear functions of $r$. Since $\varsigma(z)$ expands as
\[z+\frac{z^3}{24}+\cdots,\]
we know that each $\varsigma$ function contributes at least $1$ to the $y$-degree. So $\varsigma_{\mathcal{J}}$ contributes at least $l(P_j)-1$ to the $y$-degree. As for the term
\[\prod_{i\in P_j}\frac{1}{r+b_i}  \frac{(\mathcal{S}(ruy_i))^{m_i+\frac{b_i}{r}}}{(1+\frac{r+b_i}{r})_{m_i-1}}y_i^{m_i},\]
since $|\vec{m}_{P_j}|=\sum_i m_i=0$, it always contribute at least $0$ to the $y$-degree. So totally, each monomial in $f_{\mathcal{J}}$ has $y$-degree at least $l(P_j)-1$. And since the $y$-degree coming from
\[\left(\prod_{i\in P_j}(\mathcal{S}(ruy_i))^{m_i+\frac{b_i}{r}}\right)\varsigma_{\mathcal{J}}\]
jumps at least $2$, we know that there is no monomial with $y$-degree $l(P_j)$.

Now in order to extract the $(y_{P_j})!$-coefficient of \eqref{eqn:fmexmidterm} whose $y$-degree is $l(P_j)$, we have to borrow one $y$-degree from $\mathcal{E}_{|\vec{b}_{P_j}|}(u|y_{P_j}|)$. It is now easy to see that $(y_{P_j})!$-coefficient of \eqref{eqn:fmexmidterm} has the following form:
\[C_{\mathcal{J}}(r)u^{l(P_i)}\mathcal{E}_{|\vec{b}_{P_j}|}[0]\]
where $C_{\mathcal{J}}(r)$ is a certain rational function on $r$.
So the $(y_{P_j})!$-coefficient of  $f_{\vec{b}_{P_j}}\mathcal{E}_{|\vec{b}_{P_j}|}(u|y_{P_j}|)$ has the following form
\[\left(\sum_{|\vec{m}_{P_j}|=0}\sum_{\mathcal{J}\in  \mathcal{K}(R_{\vec{m}_{P_j},\vec{b}_{P_j}})\atop \mathcal{J}\text{ is connected}}C_{\mathcal{J}}(r)\right)u^{l(P_i)}\mathcal{E}_{|\vec{b}_{P_j}|}[0].\]
We define
\begin{equation}\label{eqn:comfm}
C_{P_j}(r)=\left(\sum_{|\vec{m}_{P_j}|=0}\sum_{\mathcal{J}\in  \mathcal{K}(R_{\vec{m}_{P_j},\vec{b}_{P_j}})\atop \mathcal{J}\text{ is connected}}C_{\mathcal{J}}(r)\right).
\end{equation}
$C_{P_j}(r)$ is still a rational function because it is a finite sum.
\end{proof}

Now we are left to determine the rational function $C_{P_j}(r)$. A priori, we have a combinatorial formula \eqref{eqn:comfm} to compute it. But it is too complicated for a real computation. So we need to find a way around it. The new way relies on some knowledge of relative invariants which will be computed in Appendix \ref{sec:wdvvcal} using the WDVV equation. We need some preparation before we compute $C_{P_j}(r)$.

\begin{defn}\label{def:extrnota1}
  Let $P$ be a partition of the set $\{1,\cdots, n\}$. Let $P_1, \cdots, P_k$ be the parts of the partition $P$ ordered by the smallest element each part contains. We further require that elements in $P_j$ are ordered from small to big for $\forall j$. Let $\vec{b}=(b_1,\cdots,b_n)$ be an element of $\Z^n$.
  We define
  \begin{equation*}
      \vec{b}_P\coloneqq(\vec{b}_{P_1}, \cdots, \vec{b}_{P_k}).
  \end{equation*}
   So $\vec{b}_P$ is an element of $\mathbb{Z}^{ l(P_1)} \times \cdots \times \mathbb{Z}^{l(P_k)}$. 
\end{defn}
	
\begin{defn}\label{def:extranota2}
		Let $\vec{b}$ and $P$ be as in Definition \ref{def:extrnota1}. We define 
		\begin{equation*}
		\mathcal{E}_{\vec{b}_P}[0]\coloneqq \prod_{j=1}^k\mathcal{E}_{|\vec{b}_{P_j}|}[0].
		\end{equation*}
We further define 
	\begin{equation*}
	   N_{\vec{b}_P}\coloneqq \prod_{j=1}^k\left( (l(P_j)-1)! \prod_{i=2}^{l(P_j)} (\vec{b}_{P_j})_i\right)
	\end{equation*}
where $(\vec{b}_{P_j})_i$ denotes the $i$-th element of $\vec{b}_{P_j}$.
\end{defn}
\begin{defn}
Let $a\in\mathbb{Z}$. We define
\[
\{\alpha_a,\mathcal{E}_{\vec{b}_P}[0]\}\coloneqq\left[\cdots\left[ \left[\alpha_a, \mathcal{E}_{|\vec{b}_{P_1}|}[0] \right ],\mathcal{E}_{|\vec{b}_{P_2}|}[0]\right ],\cdots, \mathcal{E}_{|\vec{b}_{P_k}|}[0] \right ].
\]
\end{defn}
\begin{lem}
\label{lem:comb}
Let $a$ be a positive integer and $\vec{b}=(b_1,\cdots, b_n)$ be a vector of negative integers. We further assume that $a+b_1+\cdots+b_n>0$. Then we have the following equation:
\begin{equation}\label{eqn:keycomfm}
    \sum_{P} N_{\vec{b}_p} \{\alpha_a, \mathcal{E}_{\vec{b}_P}[0] \}=a (a+b_1+\cdots+b_n)^{n-1} \alpha_{d}
\end{equation}
where we sum over all partitions $P$ of $\{1, ..., n\}$.
\end{lem}
\begin{proof}
We repeatedly apply the commutation relation \eqref{eq:commrel2} to the LHS and then compare with the expansion of the RHS.
\end{proof}

We are ready to show that
\begin{lem}\label{lem:detcoeff}
The rational function $C_{P_j}(r)$ in Lemma \ref{lem:formalneg} equals to
\[\frac{(l(P_j)-1)!\prod_{i=2}^{l(P_j)}(\vec{b}_{P_j})_i}{r^{l(P_j)}}.\]
\end{lem}

\begin{proof}
We will prove it via induction on the number of elements in $P_j$, i.e.,  $l(P_j)$. The case $l(P_j)=1$ is trivial. We assume that the lemma is true for $l(P_j)<m$. We will determine $C_{P_j}(r)$ when $l(P_j)=m$ by computing the following 
genus zero relative invariant in two different ways:
\begin{equation}\label{eqn:twowaysinv}
\left\langle\vec{\mu}_0|\emptyset|\vec{\mu}_{\infty}\right\rangle_{g=0,d}
\end{equation}
where $\vec{\mu}_0=(a,b_1,\cdots,b_m)$ such that $a>0$ and $b_i<0$ for $\forall i$, $\vec{\mu}_{\infty}=(d)$. 

The first way is using the operator formula \eqref{eqn:orbiP1}.
By \cite{FWY}, we know that genus zero relative invariant \eqref{eqn:twowaysinv} is equal to the corresponding orbifold invariants after multiplying $r^m$. The latter can be computed via taking the $x y_1\cdots y_m x't^0u^{-2}q^d$-coefficient of 
the following connected vacuum expectation when both $r$ and $s$ become sufficiently large:
\[r^m\left\langle \bA_{a/r}(x)\prod_{j=1}^m \bA_{(r+b_j)/r}(y_j) e^{\frac{t\alpha_{r}}{ru}}\left(\frac{q}{t^{1/r}(-t)^{1/s}}\right)^{H} e^{\frac{-t\alpha_{-s}}{su}}A^*_{d/r}(x')\right\rangle^{\circ}.\]
Using a similar large $r,s$ analysis as before, we know that the $x y_1\cdots y_m x't^0u^{-2}q^d$-coefficient can also be computed by
\[\frac{r^{m}}{u^{2+m}}\sum_{P'}\left\langle\frac{\alpha_a}{a}\left(\prod_{j=1}^{l(P')}C_{P'_j}(r)u^{l(P'_j)}\mathcal{E}_{|\vec{b}_{P'_j}|}[0]\right)\frac{\alpha_{-d}}{d}\right\rangle\]
where $P'$ runs over all the partitions of $\{1,\cdots,m\}$. Note that each vacuum expectation in above summation equals to its connected one.
By induction and using the notations in Definition \ref{def:extranota2}, we know that the above vacuum expectation can be further written as 
\begin{equation}\label{eqn:1stway}
\sum_{P':l(P')>1}\left\langle\frac{\alpha_a}{a}N_{\vec{b}_{P'}}\mathcal{E}_{\vec{b}_{P'}}[0]\frac{\alpha_{-d}}{d}\right\rangle+r^m\left\langle\frac{\alpha_a}{a}C_{P_j}(r)\mathcal{E}_{|\vec{b}_{P_j}|}[0]\frac{\alpha_{-d}}{d}\right\rangle
\end{equation}
where we use the notation $P_j$ appearing in the induction to denote the unique part $\{1,2,\cdots,m\}$ of the partition $P'$ s.t. $l(P')=1$. 

The second way to compute \eqref{eqn:twowaysinv} is using the WDVV equation. The result is that
\begin{equation}
\label{eqn:2ndway}
\left\langle\vec{\mu}_0|\emptyset|\vec{\mu}_{\infty}\right\rangle_{g=0,d}=(a+b_1+\cdots+b_m)^{m-1}
\end{equation}
whose computation details are given in Appendix \ref{sec:wdvvcal}. Now a direct comparison of \eqref{eqn:1stway} and \eqref{eqn:2ndway} plus Lemma \ref{lem:comb} imply that 
\[C_{P_j}(r)=\frac{(m-1)!\prod_{i=2}^m b_i}{r^m}=\frac{(l(P_j)-1)!\prod_{i=2}^{l(P_j)} (\vec{b}_{P_j})_i}{r^{l(P_j)}}.\]
So Lemma \ref{lem:detcoeff} also holds for $l(P_j)=m$. By induction, it holds for all $l(P_j)$.
\end{proof}

\begin{rmk}
At first glance, it may be strange that the formula for $C_{P_j}(r)$ is not symmetric among those $(\vec{b}_{P_j})_i$.
This asymmetry is due to the non-commutativity of the operators $\mathcal{E}_{P_j}[0]$. After summing over all different partitions, the total sum
\[\sum_P N_{\vec{b}_P}\mathcal{E}_{\vec{b}_P}[0]\]
will be symmetric. For example, when $\vec{b}$ has two elements, we have
\[b_2\mathcal{E}_{b_1+b_2}[0]+\mathcal{E}_{b_1}[0]\mathcal{E}_{b_2}[0]=b_1\mathcal{E}_{b_2+b_1}[0]+\mathcal{E}_{b_2}[0]\mathcal{E}_{b_1}[0].\]
The above equality follows from the commutation relation \eqref{eq:commrel}.
\end{rmk}

Combining \eqref{eqn:midfm-noneqcap}, Lemmas \ref{lem:formalneg} and \ref{lem:detcoeff}, it is easy to deduce the operator formula for relative invariants of the cap with negative contact:
\begin{thm}\label{thm:noneqcap}
Let $\vec{\mu}_0=(a_1, ..., a_{l_0}, b_1, ..., b_{m_0})$ be an ordered partition of a degree $d>0$ such that $a_i>0,b_j<0$, $\forall i,j$. Using the notations in Definition \ref{def:extranota2},  we have the following operator formula:
\begin{equation*}
   \left\langle\vec{\mu}_0\left|\prod_{i=1}^n \tau_{k_i}(\omega)\right.\right\rangle^{\bullet}= \frac{1}{(\prod_{i=1}^{l_0} a_i) d!} \left\langle \prod_{i=1}^{l_0} \alpha_{a_i} (\sum_{P} N_{\vec{b}_p} \mathcal{E}_{\vec{b}_P}[0] )\prod_{i=1}^n \mathcal{E}_0[k_i] (\alpha_{-1})^d\right\rangle,
\end{equation*}
where we sum over all partitions $P$ of $\{1, ..., m_0\}$ and $\vec{b}=( b_1, ..., b_{m_0})$.
\end{thm}

\begin{rmk}\label{rmk:indepofr}
Notice that by Lemmas \ref{lem:formalneg} and \ref{lem:detcoeff}, $r^{m_0}$ in 
\begin{equation*}
\frac{r^{m_0}}{u^{l_0+d+m_0+n}}\sum_P\left\langle \prod_{i=1}^{l_0}\frac{\alpha_{a_i}}{a_i}\prod_{j=1}^{l(P)}f_{\vec{b}_{P_j}}\mathcal{E}_{|\vec{b}_{P_j}|}(u|y_{P_j}|)\prod_{k=1}^{n}\mathcal{E}_0(uz_k)\frac{(\alpha_{-1})^d}{d!} \right\rangle
\end{equation*}
cancels with the corresponding powers of $r$ appearing in $f_{\vec{b}_{P_j}}$. So the calculation actually shows that 
the stationary relative invariants of $(\PP^1,0)$ are equal to (up to some powers of $r$) the corresponding stationary orbifold invariants in all genera. This recovers \cite[Theorem 1.9]{TY} in the special case when the target curve is $\PP^1$.
\end{rmk}

\subsubsection{Tube case}
By analysing both large $r$ and large $s$ behavior of the operator formula \eqref{thm:orbiP1} in the non-equivariant limit, we can similarly derive the following operator formula for the relative invariants of $(\PP^1,0\cup\infty)$ with negative contact orders:
\begin{thm}\label{thm:opform}
Let 
\[\vec{\mu}_0=(a_1, \cdots, a_{l_0}, b_1, \cdots, b_{m_0}),\quad \vec{\mu}_{\infty}=(a_1', \cdots, a'_{l_{\infty}}, b_1', \cdots, b'_{m_{\infty}})\]
be two ordered partitions of a degree $d>0$ such that $a_i,a'_i>0$, $b_j, b'_j<0$, $\forall i,j$. Using the notations in Definition \ref{def:extranota2},  we have the following operator formula:
\begin{gather*}
   \left\langle\vec{\mu}_0\left|\prod_{i=1}^n \tau_{k_i}(\omega)\right|\vec{\mu}_{\infty}\right\rangle^{\bullet}=\\ \frac{1}{\prod_{i=1}^{l_0} a_i \prod_{i=1}^{l_{\infty}} a_i'} \left\langle \prod_{i=1}^{l_0} \alpha_{a_i} (\sum_{P} N_{\vec{b}_p} \mathcal{E}_{\vec{b}_P}[0] ) \prod_{i=1}^n \mathcal{E}_0[k_i] (\sum_{P'} N_{\vec{b}'_{P'}} \mathcal{E}_{\vec{b}'_{P'}}^*[0]) \prod_{i=1}^{l_{\infty}} \alpha_{-a'_i}\right\rangle
\end{gather*}
where we sum over all partitions $P$ of $\{1, ..., m_0\}$ and all partitions $P'$ of $\{1, ..., m_{\infty}\}$, $\vec{b}=(b_1, ..., b_{m_0})$ and $\vec{b}'=(b_1', \cdots, b'_{m_{\infty}})$.
\end{thm}

\begin{rmk}
If we set $\vec{\mu}_{\infty}=(1,1,\cdots,1)$, then the operator formula for the tube will become the operator formula for the cap (up to the automorphism factor $d!$). This shows the compatibility of the two formulae.
\end{rmk}

\begin{rmk}
We notice that if there are no negative contact orders we recover (up to some automorphism factors) the operator formula in \cite{OP1}:
\begin{equation}
\label{eq:opold}
   \left\langle\vec{\mu}_0\left |\prod_{i=1}^n \tau_{k_i}(\omega)\right|\vec{\mu}_{\infty}\right\rangle^{\bullet}= \frac{1}{\prod_{i=1}^{l_0} a_i \prod_{i=1}^{l_{\infty}} a_i'} \left\langle \prod_{i=1}^{l_0} \alpha_{a_i}\prod_{i=1}^n \mathcal{E}_0[k_i]\prod_{i=1}^{l_{\infty}} \alpha_{-a'_i} \right\rangle.
\end{equation}
\end{rmk}

\section{Example}\label{sec:Ex}
In this section, we use the operator formula in Theorem~\ref{thm:opform} to derive an explicit formula for the generating function of stationary relative invariants of $\mathbb{P}^1$ with one inner marking. We write
\begin{equation}
    F_{\vec{\mu}_{0}, \vec{\mu}_{\infty}}^{\circ}(z_1, ..., z_n)= \sum_{k_1, ..., k_n\geq 0} \left\langle \vec{\mu}_{0} \left| \prod_{i=1}^n \tau_{k_i}(\omega)\right| \vec{\mu}_{\infty} \right \rangle^{\circ} \prod_{i=1}^n z_i^{k_i+1}
\end{equation}
for the generating function of connected relative invariants of $\mathbb{P}^1$. If there are no negative contact orders, then we have the following explicit formula for the generating function of stationary relative invariants of $\mathbb{P}^1$ with one inner marking by \cite{OP1}:
\begin{equation}
\label{eq:OPconnect}
    F_{\vec{\mu}_{0}, \vec{\mu}_{\infty}}^{\circ}(z)=\frac{\prod_{i=1}^{l_0} \varsigma(a_i z) \prod_{j=1}^{l_\infty} \varsigma(a'_j z) }{(\prod_{i=1}^{l_0} a_i \prod_{j=1}^{l_{\infty}} a'_j) \varsigma(z)}.
\end{equation}
Okounkov and Pandharipande prove this formula by computing the connected expectation of their operator formula Eq. \eqref{eq:opold}. We extend their formula to include negative contact.
\begin{thm}
\label{thm:explform}
We have the following formula:
\[
\begin{aligned}
    F_{\vec{\mu}_{0}, \vec{\mu}_{\infty}}^{\circ}(z)=
\frac{1}{\prod_{i=1}^{l_0} a_i \prod_{j=1}^{l_{\infty}} a'_j} \frac{1}{\varsigma(z)}\cdot\\
\left(\sum_{J_1, J_2\cdots, J_s\atop a_{l_i}+|b_{J_i}|>0}\prod_{i=1}^s a_{l_i}(a_{l_i}+|b_{J_i}|)^{l(J_i)-1}\prod_{t\neq l_1, ..., l_s} \varsigma(a_t z)\prod_{i=1}^s \varsigma((a_{l_i}+|b_{J_i}|)z)\right)\cdot\\
\left(\sum_{J'_1, J'_2\cdots, J'_{s'}\atop a'_{l'_i}+|b'_{J'_i}|>0}\prod_{i=1}^{s'} a'_{l'_i}(a'_{l'_i}+|b'_{J'_i}|)^{l(J'_i)-1}\prod_{t\neq l'_1, ..., l'_{s'}} \varsigma(a'_t z)\prod_{i=1}^{s'} \varsigma((a'_{l'_i}+|b'_{J'_i}|)z)\right)
\end{aligned}\]
where the first summation takes over all different partitions $J_1,\cdots,J_s$ of $\{1,2,\cdots,m_0\}$ and different $l_i\in \{1,2,\cdots,l_0\}$ such that $l_j<l_{j+1}$ (ordered) and $a_{l_i}+|b_{J_i}|>0$. The second summation takes over all the partitions $J'_1,\cdots,J'_{s'}$ of $\{1,2,\cdots, m_{\infty}\}$ and different $l'_i\in \{1,2,\cdots,l_\infty\}$ such that $l'_j<l'_{j+1}$ (ordered) and $a'_{l'_i}+|b_{J'_i}|>0$. Here  we recall that $|b_{J_i}|=\sum_{j\in J_i}b_{j}$ and $|b_{J'_i}|=\sum_{j\in J'_i}b'_{j}$.
\end{thm}
\begin{proof}
We notice that Theorem~\ref{thm:opform} gives the operator formula for the generating function of disconnected relative invariants of $(\mathbb{P}^1,0\cup\infty)$:
\begin{gather*}
     F_{\vec{\mu}_{0}, \vec{\mu}_{\infty}}^{\bullet}(z)=\\ \frac{1}{\prod_{i=1}^k a_i \prod_{i=1}^l a_i'} \langle \prod_{i=1}^{k} \alpha_{a_i} (\sum_{P} N_{\vec{b}_p} \mathcal{E}_{\vec{b}_P}[0] ) \mathcal{E}_0(z) (\sum_{P'} N_{\vec{b'}_{P'}} \mathcal{E}_{\vec{b'}_{P'}}^*[0]) \prod_{i=1}^l \alpha_{-a'_i} \rangle.
\end{gather*}
To compute the explicit formula for the generating function of connected invariants, we compute the connected expectation of this operator formula. To do this we recall Eq. \eqref{eq:commrel2}: For $a$ and $b$ nonzero integers we have 
\begin{equation}\label{eq:commrelToo}
    [\alpha_{a}, \mathcal{E}_{b}[0]]=
    \begin{cases}
    a \alpha_{a+b}, \text{ if } a+b\neq0 \\
    a C, \text{ if } \ a+b=0\\
    \end{cases}.
\end{equation}
where $C$ is the charge operator. In our operator formula, we are going to commute the $\mathcal{E}_{b_i}[0]$ operators to the left for $i=1, ..., m_0$ resp. the $\mathcal{E}^*_{b'_j}[0]$ operators to the right for $j=1, ..., m'_0$. Equation \eqref{eq:commrelToo} determines the interactions of $\mathcal{E}_b[0]$ operators with $\alpha_a$ operators. We can ignore the cases where $a+b=0$ since the charge operator will kill the operator expectation. Similarly we can ignore the cases where $a+b<0$ since in this case we can move the $\alpha_{a+b}$ to the left. If we happen to commute with an operator $\alpha_{-(a+b)}$ we get a constant factor. These terms we can ignore because it corresponds to a disconnected commutation pattern. When $\alpha_{a+b}$ arrives at the left end it will kill the operator expectation. This means we only have to consider interactions where $a+b>0$. 
\begin{gather*}
   \langle \prod_{i=1}^{k} \alpha_{a_i} (\sum_{P} N_{\vec{b}_p} \mathcal{E}_{\vec{b}_P}[0] ) \mathcal{E}_0(z) ...\rangle^{\circ}=\\
   \sum_{J_1, J_2\cdots, J_s\atop a_{l_i}+|b_{J_i}|>0} \langle \prod_{i\neq l_1, ..., l_s}\alpha_{a_i}\prod_{k=1}^s\left (\sum_{P_k\preceq J_k} N_{\vec{b}_{P_k}}\{\alpha_{l_k}, \mathcal{E}_{\vec{b}_{P_k}}[0]\} \right )\mathcal{E}_0(z)...\rangle^{\circ}=\\
   \sum_{J_1, J_2\cdots, J_s\atop a_{l_i}+|b_{J_i}|>0} \prod_{i=1}^s a_{l_i}(a_{l_i}+|b_{J_i}|)^{|J_i|-1} \langle \prod_{i\neq l_1, ..., l_s}\alpha_{a_i}\prod_{k=1}^s \alpha_{a_{l_k}+|b_{J_k}|}\mathcal{E}_0(z)...\rangle^{\circ},
\end{gather*}
where we sum over all partitions $J_1, ..., J_s$ of $\{1, ..., m_0\}$ and $1\leq l_1 < l_2 < ... < l_s\leq l_0$ such that $a_{l_i}+|b_{J_i}|>0$ for $i=1, ..., s$. Note that we have used Lemma~\ref{lem:comb} in the second equality. We can do the same calculation on the RHS of $\mathcal{E}_0(z)$ in the operator expectation. The connected expectation of the remaining terms can be computed using Eq.\eqref{eq:OPconnect} and we get our claimed formula.
\end{proof}
\begin{rmk}
By \cite{FWY2}, we know that relative invariants with negative contact can also be computed using ordinary relative invariants and rubber invariants. More precisely we sum over certain bipartite graphs where the two sides correspond to the relative side and rubber side. For each bipartite graph, we get a fiber product of the corresponding virtual classes of relative and rubber moduli spaces. The resulting class for each of these bipartite graphs is then capped with an obstruction class. If we consider relative invariants appearing in Theorem \ref{thm:explform}, this obstruction class will force only certain types of bipartite graphs to appear.

\includegraphics[width=\textwidth]{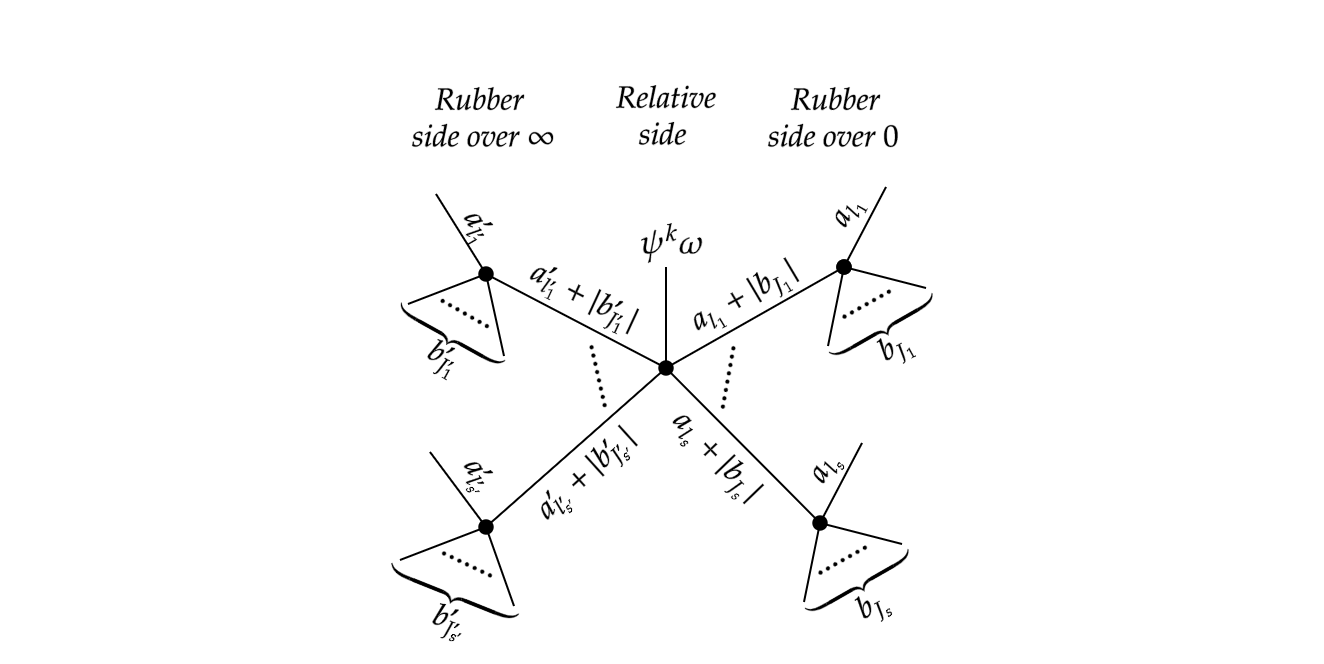}
Namely the bipartite graphs of the form shown in this picture. Note that all genus is concentrated in the unique vertex lying in the relative side. The $i$-th vertex on the rubber side over $0$ (resp. over $\infty$) will contribute a factor
$(a_{l_i}+|b_{J_i}|)^{l(J_i)}$ (resp. $(a'_{l'_i}+|b'_{J'_i}|)^{l(J'_i)}$). And the unique vertex on the relative side will contribute the factor
\[\frac{\prod_{t\neq l_1, ..., l_s} \varsigma(a_t z)\prod_{i=1}^s \varsigma((a_{l_i}+|b_{J_i}|)z)\prod_{t\neq l'_1, ..., l'_{s'}} \varsigma(a'_t z)\prod_{i=1}^{s'} \varsigma((a'_{l'_i}+|b'_{J'_i}|)z)}{\varsigma(z)\prod_{t\neq l_1, ..., l_s} a_t\prod_{i=1}^s(a_{l_i}+|b_{J_i}|)\prod_{t\neq l'_1, ..., l'_{s'}}a'_t\prod_{i=1}^{s'}(a'_{l'_i}+|b'_{J'_i}|)}\]
By summing over all different bipartite graphs, we derive the formula in Theorem \ref{thm:explform} again.
\end{rmk}

\appendix
\section{Lie brackets and interaction diagrams}\label{sec:interaction}
In this Appendix, we want to prove our key lemma in Section \ref{sec:opfm}, i.e., Lemma \ref{lem:basic}. The key lemma gives a uniform way to compute 
\begin{equation*}
\langle L_1 \prod_i \boldsymbol{A}_{a_i}(z_i,u,t)  L_2\rangle
\end{equation*}
when $r$ becomes very large. Recall that $L_1$ and $L_2$ are two operators with fixed energies. By Definition \ref{def:A-ops} and \eqref{eqn:A-new}, we know that each $\boldsymbol{A}_{a_i}$ can be uniquely decomposed as a summation of pure energy operators with energy in the form of $-m_i r-a_i$, $m_i\in \Z$. We use $R_{m_i,a_i}$ to denote the pure energy operator in the unique decomposition of $\boldsymbol{A}_{a_i}$ with energy $-m_i r-a_i$. So its $r$-energy is $-m_i$. For example, if $a_i>0$, then we have
\[R_{m_i,a_i}=\frac{t^{\frac{a_i}{r}}}{u}\frac{z_i}{tz_i+a_i} \mathcal{S}(ruz_i)^{\frac{tz_i+a_i}{r}} \frac{(tz_i\mathcal{S}(ruz_i))^{m_i}}{(1+\frac{tz_i+a_i}{r})_{m_i}} \mathcal{E}_{m_i r+a}(uz_i).\]
So we only need to compute the vacuum expectation of the operators in the following form:
\begin{equation}\label{eqn:vacexpapp}
\left\langle L_1\prod_{i} R_{m_i,a_i} L_2\right\rangle.
\end{equation}
Note that if the summation $\sum_{i} m_i\neq 0$, then the total $r$-energy of the operator $L_1\prod_{i} R_{m_i,a_i} L_2$ is nonzero because $L_1$ and $L_2$ have fixed energies (do not depend on $r$). Then the energy of $L_1\prod_{i} R_{m_i,a_i} L_2$ is nonzero when $r$ becomes sufficiently large. This implies that the expectation of $L_1\prod_{i} R_{m_i,a_i} L_2$ must be zero for large $r$. So we may assume that $\sum_i m_i=0$. We will give a recursive way to compute \eqref{eqn:vacexpapp}. To illustrate our general procedure, let us first see an example.

\begin{ex}\label{ex:app1}
We want to compute
\begin{equation}\label{eqn:exappex}
\left\langle L_1 R_{1,a_1}R_{0,a_2}R_{0,a_3}R_{-1,a_4} L_2\right\rangle
\end{equation}
when $r$ becomes sufficiently large.
Using the commutators, we can move the positive $r$-energy operator $R_{-1,a_4}$ to the left:
\begin{equation}\label{eqn:exmidterm1}
\begin{aligned}
\left\langle L_1 [R_{1,a_1},R_{-1,a_4}]R_{0,a_2}R_{0,a_3} L_2\right\rangle +&\left\langle L_1 R_{1,a_1}[R_{0,a_2},R_{-1,a_4}]R_{0,a_3} L_2\right\rangle+\\
\left\langle L_1 R_{1,a_1}R_{0,a_2}[R_{0,a_3},R_{-1,a_4}] L_2\right\rangle +&\left\langle L_1 R_{-1,a_4}R_{1,a_1}R_{0,a_2}R_{0,a_3} L_2\right\rangle.
\end{aligned}
\end{equation}
Let us first see the last term $\left\langle L_1 R_{-1,a_4}R_{1,a_1}R_{0,a_2}R_{0,a_3} L_2\right\rangle$. It equals to
\[(R_{1,a_1}R_{0,a_2}R_{0,a_3} L_2 v_{\emptyset}, R_{-1,a_4}^*L_1^*v_{\emptyset})\]
where $L_1^*$, $R_{-1,a_4}^*$ are the dual operators of $L_1$, $R_{-1,a_4}$ respectively. The $r$-energy of $R_{-1,a_4}^*$ is negative. So when $r$ becomes sufficiently large, the energy of $R_{-1,a_4}^*L_1^*$ is negative. Then $R_{-1,a_4}^*L_1^*v_{\emptyset}=0$ because $v_{\emptyset}$ has the lowest energy. 

The $r$-energy of each operator in the first term 
$$\left\langle L_1 [R_{1,a_1},R_{-1,a_4}]R_{0,a_2}R_{0,a_3} L_2\right\rangle$$
of \eqref{eqn:exmidterm1} is zero. We leave it unchanged. 

As for the second term of \eqref{eqn:exmidterm1}
\begin{equation}\label{eqn:exmidterm2}
\left\langle L_1 R_{1,a_1}[R_{0,a_2},R_{-1,a_4}]R_{0,a_3} L_2\right\rangle,
\end{equation}
the energy of the bracket $[R_{0,a_2},R_{-1,a_4}]$ is the summation of the energy of $R_{0,a_2}$ and $R_{-1,a_4}$. So the $r$-energy of the bracket is $1$. We can use the commutator to further move the bracket to the left. Then \eqref{eqn:exmidterm2} becomes
\[\left\langle L_1 \left[R_{1,a_1},[R_{0,a_2},R_{-1,a_4}]\right]R_{0,a_3} L_2\right\rangle+\left\langle L_1 [R_{0,a_2},R_{-1,a_4}]R_{1,a_1}R_{0,a_3} L_2\right\rangle.\]
Since the $r$-energy of the commutator $[R_{0,a_2},R_{-1,a_4}]$ is positive, using a similar analysis for the last term in \eqref{eqn:exmidterm1}, we know that the second term
\[\left\langle L_1 [R_{0,a_2},R_{-1,a_4}]R_{1,a_1}R_{0,a_3} L_2\right\rangle\]
vanishes when $r$ become sufficiently large.

As for the third term of \eqref{eqn:exmidterm1}
\begin{equation*}
\left\langle L_1 R_{1,a_1}R_{0,a_2}[R_{0,a_3},R_{-1,a_4}] L_2\right\rangle,
\end{equation*}
we always move the positive $r$-energy operator to the left using the commutators, which yields
\[\left\langle L_1 \left[R_{1,a_1},[R_{0,a_3},R_{-1,a_4}]\right]R_{0,a_2} L_2\right\rangle+\left\langle L_1 \left [R_{1,a_1},\left[R_{0,a_2},[R_{0,a_3},R_{-1,a_4}]\right]\right] L_2\right\rangle.\]
Finally, we may conclude that when $r$ becomes sufficiently large \eqref{eqn:exappex} equals to 
\[
\begin{aligned}
\left\langle L_1 [R_{1,a_1},R_{-1,a_4}]R_{0,a_2}R_{0,a_3} L_2\right\rangle &+\left\langle L_1 \left[R_{1,a_1},[R_{0,a_2},R_{-1,a_4}]\right]R_{0,a_3} L_2\right\rangle\\
\left\langle L_1 \left[R_{1,a_1},[R_{0,a_3},R_{-1,a_4}]\right]R_{0,a_2} L_2\right\rangle & +\left\langle L_1 \left [R_{1,a_1},\left[R_{0,a_2},[R_{0,a_3},R_{-1,a_4}]\right]\right] L_2\right\rangle.
\end{aligned}
\]
\end{ex}
Note that in the above example, we always move positive $r$-energy operators to the left until all the operators have zero $r$-energy. This leads us to the following definition:

\begin{defn}
Let $\mathcal{O}_1,\cdots,\mathcal{O}_n$ be operators on $\Lambda ^{\frac{\infty}{2}}_0 V$ 
with energies given by $k_1r+b_1,\cdots,k_nr+b_n$ respectively. Assume that $\sum_i k_i=0$. We define
$F_n(\mathcal{O}_1,\mathcal{O}_2,\cdots,\mathcal{O}_n)$ recursively: if $k_1>0$, then $F_n=0$; if $k_1=\cdots=k_n=0$, then $F_n=\prod_i \mathcal{O}_i$; Otherwise, let $i_0$ be the smallest integer with $k_{i_0}>0$, then
\[F_n(\mathcal{O}_1,\cdots,\mathcal{O}_n)=\sum_{j=1}^{i_0-1}F_{n-1}(\mathcal{O}_1,\cdots,\hat{\mathcal{O}}_j,[\mathcal{O}_j,\mathcal{O}_{i_0}],\cdots,\hat{\mathcal{O}}_{i_0},\cdots, \mathcal{O}_n)\]
where the hat indicates that we omit the corresponding term.
\end{defn}

Note that we could reinterpret the previous example as saying
\[
\left\langle L_1 R_{1,a_1}R_{0,a_2}R_{0,a_3}R_{-1,a_4} L_2\right\rangle=\left\langle L_1 F_4(R_{1,a_1},R_{0,a_2},R_{0,a_3},R_{-1,a_4}) L_2\right\rangle
\]
for sufficiently large $r$. And it is easy to show inductively that if there exists an integer $1\leq j_0\leq n$ such that $\sum_{i=1}^{j_0}k_i>0$, then $F_n(\mathcal{O}_1,\cdots,\mathcal{O}_n)=0$.

In order to keep track of the concatenated Lie bracket expressions appearing in
the recursive definition of $F_n$, we need to introduce the so-called interaction diagrams. The notation of interaction diagrams was first introduced in \cite{OP3} for similar reasons.

\begin{defn}
We call a directed graph $\mathcal{J}$ with vertex set $V\subset \mathbb{Z}_{>0}$ an \emph{interaction diagram} if it satisfies the following properties:\\
(i) $\lvert V \rvert< \infty$.\\
(ii) For every edge $(j,i)$\footnote{The pair $(j,i)$ stands for a directed edge which starts from the vertex $j$ and ends on $i$.} we have $j>i$ and for every vertex there is at most one outgoing edge.
\end{defn}
Note that from the condition (ii), we can easily deduce that $\mathcal{J}$ is a forest, i.e. each connected component is a tree.

Using the interaction diagrams, we can record the recursive process in the definition of $F_n$ as follows. Each summand in the final summation of $F_n$ corresponds to an interaction diagram. The vertex set is $\{1,2,\cdots,n\}$ which are in one-to-one correspondence with the $n$ inputs $\mathcal{O}_1,\cdots,\mathcal{O}_n$. The edges record how these operators interact with each other. More precisely, if the operator at vertex $a$ interacts with the operator at vertex $b$ with $a>b$, then we draw a directed edge from $a$ to $b$. 

For example, in the first step, we choose the smallest vertex $i_0$ such that the $r$-energy of $\mathcal{O}_{i_0}$ is positive, and then interact $\mathcal{O}_{i_0}$ with some operator $\mathcal{O}_j$ before it. Correspondingly, we draw a directed edge from $i_0$ to $j$. In the second step, we omit the vertex $i_0$ and replace the corresponding operator $\mathcal{O}_j$ by $[\mathcal{O}_j,\mathcal{O}_{i_0}]$ at vertex $j$. We then choose the smallest vertex $i_0'$ among the remaining $n-1$ vertices
whose corresponding operator has positive $r$-energy. The operator at vertex $i_0'$ will then interact with some operator at vertex $j'$ with $j'<i_0'$. We then draw a directed edge from $i_0'$ to $j'$ correspondingly. This algorithm gives us an interaction diagram for each summand in $F_n$.

\begin{ex}\label{ex:app2}
From the Example \ref{ex:app1}, we know that four summands are created in the algorithm of $F_4(R_{1,a_1},R_{0,a_2},R_{0,a_3},R_{-1,a_4})$:
\[
\begin{aligned}
(i)\qquad\;\; [R_{1,a_1},R_{-1,a_4}]R_{0,a_2}R_{0,a_3} , &\quad\quad  (ii)\quad\left[R_{1,a_1},[R_{0,a_2},R_{-1,a_4}]\right]R_{0,a_3}, \\
(iii)\quad \left[R_{1,a_1},[R_{0,a_3},R_{-1,a_4}]\right]R_{0,a_2}, &\quad\quad  (iv)\quad \left [R_{1,a_1},\left[R_{0,a_2},[R_{0,a_3},R_{-1,a_4}]\right]\right].
\end{aligned}
\]
Their corresponding interaction diagrams are given respectively by
\begin{equation*}
\xymatrix{
\\
\qquad \bullet_{1} & \bullet_{2} & \bullet_{3} & \bullet_{4}
\ar@/_2pc/[lll]\\
\bullet_{1} & \bullet_{2} \ar@/_2pc/[l] & \bullet_{3} & \bullet_{4} \ar@/_2pc/[ll] \\
\bullet_{1} & \bullet_{2} & \bullet_{3} \ar@/_2pc/[ll] & \bullet_{4} \ar@/_2pc/[l] \\
\bullet_{1} & \bullet_{2} \ar@/_2pc/[l] & \bullet_{3} \ar@/_2pc/[l] & \bullet_{4} \ar@/_2pc/[l] \\}
\end{equation*}
\end{ex}

Conversely, given an interaction diagram $\mathcal{J}$ with $n$ vertices, we could read off the corresponding summand in $F_n(\mathcal{O}_1,\cdots,\mathcal{O}_n)$ as follows.
\begin{enumerate}
    \item[(1)] We start with the product $\prod_{l=1}^n \mathcal{O}_{l}$. We search for the smallest vertex $a$ of $\mathcal{J}$ with exactly one outgoing edge and no incoming edges and look at its unique outgoing edge $(a,b)$. Then in the product we switch $\mathcal{O}_{a}$ to the right side of $\mathcal{O}_{b}$ and replace the $\mathcal{O}_{b} \mathcal{O}_{a}$ with $[\mathcal{O}_{b}, \mathcal{O}_{a}]$. So after the first step our product looks like 
\begin{equation*}
    \mathcal{O}_{1}\cdot ....\cdot \mathcal{O}_{b-1} [\mathcal{O}_{b},\mathcal{O}_{a}] \mathcal{O}_{{b+1}}\cdot ... \cdot \mathcal{O}_{{a-1}} \mathcal{O}_{{a+1}}\cdot ...\cdot \mathcal{O}_{n}.
\end{equation*}
    \item[(2)] We contract the edge $(a,b)$ along its direction. So the resulting graph does not have the vertex $a$ anymore. We think of the operator $[\mathcal{O}_{b}, \mathcal{O}_{a}]$ being assigned to the vertex $b$.  With this resulting graph and product we go back to step (1). The recursion stops when there are no edges anymore.
\end{enumerate}

\begin{defn}\label{def:nestedLie}
Given an interaction diagram $\mathcal{J}$, we use $L(\mathcal{J})$ to denote the operator created by the previously described algorithm.
\end{defn}
\begin{rmk}
The notation $L(\mathcal{J})$ does not include the important information of which operator corresponds to which vertex. We assume this is clear from the context.
\end{rmk}
In Example \ref{ex:app2}, we have seen that not every interaction diagram will appear in $F_n(\mathcal{O}_1,\cdots,\mathcal{O}_n)$.
\begin{defn}
We use $\mathcal{K}(\mathcal{O}_1,\cdots,\mathcal{O}_n)$ to denote all the interaction diagrams corresponding to the summands which are created by the algorithm of $F_n(\mathcal{O}_1,\cdots,\mathcal{O}_n)$.
\end{defn}
Clearly, we have 
\begin{equation}\label{eqn:graphsumapp}
F_n(\mathcal{O}_1,\cdots,\mathcal{O}_n)=\sum_{\mathcal{J}\in \mathcal{K}(\mathcal{O}_1,\cdots,\mathcal{O}_n)} L(\mathcal{J}).
\end{equation}

Next, we want to describe the set $\mathcal{K}(\mathcal{O}_1,\cdots,\mathcal{O}_n)$ more precisely.
\begin{defn}\label{def:valstepapp}
We call a step in the algorithm of $L(\mathcal{J})$ \emph{valid} if:
\begin{enumerate}
\item[(a)] the first factor in the product has non-positive $r$-energy;
\item[(b)] the factor that is chosen to be switched to the front is the first factor in the product with positive $r$-energy.
\end{enumerate}
\end{defn}
\begin{rmk}\label{rmk:valstepapp}
Actually, for each step in the algorithm of $L(\mathcal{J})$, the factor that is chosen to be switched to the front is never the first factor. So actually condition (b) implies condition (a). Here we still include the condition (a) so as to match the algorithm of $F_n(\mathcal{O}_1,\cdots,\mathcal{O}_n)$.
\end{rmk}

\begin{defn}
We call $\mathcal{J}$ \emph{valid} for $(\mathcal{O}_1, ..., \mathcal{O}_n)$ if each step in the algorithm of $L(\mathcal{J})$ is valid and at the end of the algorithm each vertex corresponds to an operator with $r$-energy $0$.
\end{defn}

It is easy to see from the the algorithm of $F_n(\mathcal{O}_1,\cdots,\mathcal{O}_n)$ that
\begin{lem}\label{lem:valapp}
\[\mathcal{J}\in \mathcal{K}(\mathcal{O}_1,\cdots,\mathcal{O}_n)\iff \mathcal{J}\text{ is valid for }(\mathcal{O}_1, ..., \mathcal{O}_n).\]
\end{lem}
\begin{defn}\label{def:G-fun}
 We define 
   \begin{equation*}
   G_n(\mathcal{O}_1, ..., \mathcal{O}_n)\coloneqq \sum_{\substack{
   \mathcal{J}\in \mathcal{K}(\mathcal{O}_1, ..., \mathcal{O}_n)\\
   \mathcal{J} \, \mathrm{ is\, connected}
   }} L(\mathcal{J}).
   \end{equation*}
\end{defn}

\begin{ex}
From the Example \ref{ex:app2}, we can easily see that
\[G_4(R_{1,a_1},R_{0,a_2},R_{0,a_3},R_{-1,a_4})=\left [R_{1,a_1},\left[R_{0,a_2},[R_{0,a_3},R_{-1,a_4}]\right]\right].\]
\end{ex}

\begin{rmk}\label{rmk:Gfunapp}
Each operator $R_{m_i,a_i}$ has the form $c_{m_i,a_i}\mathcal{E}_{m_ir+a_i}(uz_i)$, where $c_{m_i,a_i}$ is a certain function. Then by the commutation relation \eqref{eq:commrel}, we know that $G_s(R_{m_1,a_1},\cdots,R_{m_s,a_s})$ can always be written as 
\[c_s\mathcal{E}_{\sum_i a_i}(u\sum_{i}z_i)\]
where $c_s$ is a function determined by $c_{m_i,a_i}$ and the commutation relation \eqref{eq:commrel}. Recall that we always assume that the total $r$-energy is zero. That is why the energy of $G_s(R_{m_1,a_1},\cdots,R_{m_s,a_s})$ is $-\sum_i a_i$ (independent of $r$).
\end{rmk}
\begin{defn}
Let $\mathcal{J}$ be an interaction diagram on the vertices $\{1, ..., n\}$. Let $\mathcal{J}_1, ..., \mathcal{J}_t$ be the connected components of $\mathcal{J}$ ordered by the smallest vertex each component contains. This induces a partition of $\{1, ..., n\}$. We call this partition $P$ the combinatorial type of $\mathcal{J}$.
\end{defn}
Let $P=P_1\sqcup\cdots\sqcup P_t$ be such a partition. It naturally induces a partition of the operators $(\mathcal{O}_1,\cdots,\mathcal{O}_n)$. We use $\mathcal{O}_{P_i}$ to denote the part corresponding to $P_i$. We have the following important lemma about the set $\mathcal{K}(\mathcal{O}_1,\cdots,\mathcal{O}_n)$.

\begin{lem}\label{lem:keylmapp}
Let $\mathcal{J}_1, ..., \mathcal{J}_t$ be the connected components of an interaction diagram $\mathcal{J}$. Then
\begin{equation}\label{eqn:lemapp}
\mathcal{J}\in \mathcal{K}(\mathcal{O}_1,\cdots,\mathcal{O}_n)\iff \mathcal{J}_i\in \mathcal{K}(\mathcal{O}_{P_i})\quad \forall i.
\end{equation}
\end{lem}

\begin{proof}
Let us first show $"\Rightarrow"$: By Lemma \ref{lem:valapp}, we only need to show that each $\mathcal{J}_i$ is valid. Since $\mathcal{J}_i$ is connected, we know that at the end of the algorithm of $L(\mathcal{J}_i)$, there is only one vertex. And this vertex corresponds to an operator with $r$-energy 0 since $\mathcal{J}$ is valid. We now show each step in the algorithm of $L(\mathcal{J}_i)$ is valid. If we go through every step of the algorithm of $L(\mathcal{J})$ and we ignore every step which does not use a vertex from $\mathcal{J}_i$ we exactly isolate the steps of $L(\mathcal{J}_i)$ in the right order. So we assume that we are at an arbitrary step in the algorithm of $L(\mathcal{J}_i)$. As a valid step in $L(\mathcal{J})$, we know that it is a step in $L(\mathcal{J})$ that satisfies (b) of Definition \ref{def:valstepapp}. So it must also be a step in $L(\mathcal{J}_i)$ satisfying (b). By Remark \ref{rmk:valstepapp}, this arbitrary step is a valid step for $L(\mathcal{J}_i)$.

Next, we show $"\Leftarrow"$: Because each connected component $\mathcal{J}_i$ of $\mathcal{J}$ is valid, we know that at the end of the algorithm of $L(\mathcal{J})$, each vertex must correspond to an operator with zero $r$-energy. We are left to show each step in the algorithm of $L(\mathcal{J})$ is valid. Let us assume we are at an arbitrary step in the algorithm of $L(\mathcal{J})$. Let us call our interaction diagram at this step $\mathcal{J}'$. The connected components $\mathcal{J}'_i$ for $i=1, ..., t$ are the connected components of $\mathcal{J}'$ ordered by the smallest vertex each component contains. Let $v$ be the smallest vertex of $\mathcal{J}'$ which has exactly one outgoing edge and no incoming edges. Let $v_i$ be the smallest vertex with the same property for $\mathcal{J}'_i$ for $i=1, ..., t$. It is clear that we have:
\begin{equation}\label{eqn:midapp}
    v=\min \{v_1, ..., v_t\}.
\end{equation}
We know that each $\mathcal{J}'_i$ is valid for the appropriate operators for $i=1, ..., t$. This means the first step of the algorithm of $L(\mathcal{J}'_i)$ satisfies (b) of Definition \ref{def:valstepapp} for $i=1, ...,t$. This means that $v_i$ is the smallest vertex of $\mathcal{J}'_i$ whose corresponding operator has positive $r$-energy for $i=1, ..., t$. By Equation \eqref{eqn:midapp}, we know that the first step of $L(\mathcal{J}')$ also satisfies (b). By Remark \ref{rmk:valstepapp}, we know that each step of $\mathcal{J}$ is valid.
\end{proof}

Note that we could easily express $L(\mathcal{J})$ in terms of those connected ones $L(\mathcal{J}_i)$:
\begin{equation}\label{eqn:disconn2conn}
    L(\mathcal{J})=L(\mathcal{J}_1)\cdot ... \cdot L(\mathcal{J}_t).
\end{equation}
We are ready to prove Lemma \ref{lem:basic}.
\begin{proof}[Proof of Lemma \ref{lem:basic}]
By the analysis in the beginning of this Appendix, we know that for sufficiently large $r$,
\begin{equation}\label{eqn:1app}
\langle L_1 \prod_{i=1}^n \boldsymbol{A}_{a_i}(z_i,u,t)  L_2\rangle =\sum_{\sum m_i=0}\langle L_1 F_n(R_{m_1,a_1},\cdots,R_{m_n,a_n})  L_2\rangle.
\end{equation}
So we only need to analysis the operator
\begin{equation*}
\sum_{\sum m_i=0} F_n(R_{m_1,a_1},\cdots,R_{m_n,a_n})
\end{equation*}
By \eqref{eqn:graphsumapp}, it equals to
\begin{equation}\label{eqn:midtermapp}
\sum_{\sum m_i=0}\;\sum_{ \mathcal{J}\in \mathcal{K}(R_{m_1,a_1},\cdots,R_{m_n,a_n})}L(\mathcal{J}).
\end{equation}
Each interaction diagram $\mathcal{J}$ with a fixed combinatorial type $P=P_1\sqcup\cdots\sqcup P_{l(P)}$ can be decomposed into connected components $\mathcal{J}_1,\cdots,\mathcal{J}_{l(P)}$. It naturally induces two partitions:
\[\vec{m}_{P_1}\sqcup\cdots\sqcup\vec{m}_{P_{l(P)}}=\{m_1,\cdots,m_n\},\quad\vec{a}_{P_1}\sqcup\cdots\sqcup\vec{a}_{P_{l(P)}}=\{a_1,\cdots,a_n\}\]
and a partition of operators
\[R_{\vec{m}_{P_1},\vec{a}_{P_1}}\sqcup\cdots\sqcup R_{\vec{m}_{P_{l(P)}},\vec{a}_{P_{l(P)}}}=\{R_{m_1,a_1},\cdots,R_{m_n,a_n}\}.\]
Then we may further write \eqref{eqn:midtermapp} as
\[\sum_{\sum m_i=0}\;\sum_{P=P_1\sqcup\cdots\sqcup P_{l(P)}}\sum_{\mathcal{J}_i\in\mathcal{K}(R_{\vec{m}_{P_i},\vec{a}_{P_i}})\atop \mathcal{J}_i\,\mathrm{ connected}} L(\mathcal{J}_1)\cdots L(\mathcal{J}_{l(P)}).\]
Here we have used the equation \eqref{eqn:disconn2conn} and Lemma \ref{lem:keylmapp}. By exchanging the summation order, it can be further written as
\[\sum_{P=P_1\sqcup\cdots\sqcup P_{l(P)}}\prod_{i=1}^{l(P)}\left(\sum_{|\vec{m}_{P_i}|=0}\sum_{\mathcal{J}_i\in\mathcal{K}(R_{\vec{m}_{P_i},\vec{a}_{P_i}})\atop \mathcal{J}_i\,\mathrm{ connected}}L(\mathcal{J}_i)\right).\]
Each factor in the product equals to
\[\sum_{|\vec{m}_{P_i}|=0} G_{l(P_i)}(R_{\vec{m}_{P_i},\vec{a}_{P_i}})\]
by Definition \ref{def:G-fun}. By Remark \ref{rmk:Gfunapp}, it can be written in the following form:
\[f_{\vec{a}_{P_i}}(z_{P_i},u,t,r)\mathcal{E}_{|\vec{a}_{P_i}|}(u|z_{P_i}|)\]
where $f_{\vec{a}_{P_i}}(z_{P_i},u,t,r)$ is a function indexed by the ordered set $\vec{a}_{P_i}$.

From the above discussion, we may conclude that
\begin{equation*}
\langle L_1 \prod_{i=1}^n\boldsymbol{A}_{\vec{a}} L_2\rangle=\sum_{P}\left \langle L_1  \left(\prod_{i=1}^{l(P)}f_{\vec{a}_{P_i}}(z_{P_i},u,t,r)\mathcal{E}_{|\vec{a}_{P_i}|}(u|z_{P_i}|)\right)  L_2\right\rangle
\end{equation*}
for sufficiently large $r$ and these functions $f_{\vec{a}_{P_i}}$ can be determined by the equation
\begin{equation}\label{eqn:unifunapp}
f_{\vec{a}_{P_i}}(z_{P_i},u,t,r)\mathcal{E}_{|\vec{a}_{P_i}|}(u|z_{P_i}|)=\sum_{|\vec{m}_{P_i}|=0} G_{l(P_i)}(R_{\vec{m}_{P_i},\vec{a}_{P_i}}).
\end{equation}
Obviously, these functions $f_{\vec{a}_{P_i}}$ do not depend on the operators $L_1,L_2$ but only on the ordered set $\vec{a}_{P_i}$.

\end{proof}

\section{A computation using WDVV equation}\label{sec:wdvvcal}
In this appendix, we will use the WDVV equation in \cite[Proposition 7.5]{FWY} to calculate relative invariants of $(\PP^1,0\cup \infty)$ of a certain type. Such invariants will be used in the computation of Section \ref{sec:opfm}.

To state the WDVV equation in our case, we need some preparation. First, we choose the basis $\{1, H\}$ for $H^*(\mathbb{P}^1)$ where $H$ is the hyperplane class. As for $0 \cup \infty$, we use $\boldsymbol{0}$ (resp. $ \boldsymbol{\infty}$) to denote the identity class of $H^*(0)$ (resp. $H^*(\infty)$). Let $\mathfrak{H}_0=H^*(\mathbb{P}^1)$ and $\mathfrak{H}_i=H^*(0\cup \infty)=H^*(0)\oplus H^*(\infty)$
for $i\neq 0$. The space of primary insertions for relative invariants of $(\PP^1,0\cup \infty)$ is 
\[\mathfrak{H}\coloneqq \bigoplus_{i\in\Z}\mathfrak{H}_i.\]
For an element $\alpha\in \mathfrak{H}_i$, we denote by $[\alpha]_i$ the image in $\mathfrak{H}$ via the obvious embedding. There is a natural basis for $\mathfrak{H}$:
\begin{align*}
    \tilde{T}_{0,0}&\coloneqq [1]_0,\\
    \tilde{T}_{0,1}&\coloneqq [H]_0,\\
    \tilde{T}_{i, 0}&\coloneqq [\boldsymbol{0}]_i,\; i\neq 0,\\
    \tilde{T}_{i, \infty}&\coloneqq [\boldsymbol{\infty}]_i,\; i\neq 0
\end{align*}
and a natural pairing:
 \begin{equation*}
        ([\alpha]_i, [\beta]_j)=\begin{cases}
        0,\text{  if } i+j\neq 0,\\
        \int_{\PP^1} \alpha \cup \beta,\text{  if } i=j=0,\\
        \int_{0\cup\infty} \alpha \cup \beta,\text{  if } i+j=0,\; i, j \neq 0.
        \end{cases}
    \end{equation*}
Let $\{\tilde{T}_{i,k}^{\vee}\}$ be the dual basis of $\{\tilde{T}_{i, k}\}$ under the above pairing. More precisely, we have
\begin{align*}
    \tilde{T}_{0,0}^{\vee}&\coloneqq [H]_0,\\
    \tilde{T}_{0,1}^{\vee}&\coloneqq [1]_0,\\
    \tilde{T}_{i, 0}^{\vee}&\coloneqq [\boldsymbol{0}]_{-i},\; i\neq 0,\\
    \tilde{T}_{i, \infty}^{\vee}&\coloneqq [\boldsymbol{\infty}]_{-i},\; i\neq 0.
\end{align*}

Let $[\alpha_1]_{i_1},\cdots,[\alpha_n]_{i_n}\in \{\tilde{T}_{i, k}\}$. We define
\begin{equation}\label{eqn:newepappB}
I_d([\alpha_1]_{i_1},[\alpha_2]_{i_2},\cdots,[\alpha_n]_{i_n})
\end{equation}
to be the following genus zero relative invariants
\begin{equation}\label{eqn:g0appB}
\left\langle\vec{\mu}_0\left|\prod_{s\in S}\tau_0(\alpha_s)\right|\vec{\mu}_{\infty}\right\rangle_{g=0,d}
\end{equation}
where the contact orders and the set of inner markings can be naturally read from  \eqref{eqn:newepappB}:
\[\vec{\mu}_0=\{i_s|\alpha_s=\boldsymbol{0}\},\quad \vec{\mu}_{\infty}=\{i_s|\alpha_s=\boldsymbol{\infty}\},\quad S=\{s|i_s=0\}.\]
Note that
the express \eqref{eqn:g0appB} makes sense only if $\vec{\mu}_0$ and $\vec{\mu}_{\infty}$ are two partitions of $d$. So if this is not the case, we simply set $I_d([\alpha_1]_{i_1},[\alpha_2]_{i_2},\cdots,[\alpha_n]_{i_n})=0$. Another constraint comes from the dimension constraint of \eqref{eqn:g0appB}. So if the degree summation of all the insertions in \eqref{eqn:g0appB} does not match with the virtual dimension of the corresponding genus zero relative moduli space, then we also have $I_d([\alpha_1]_{i_1},[\alpha_2]_{i_2},\cdots,[\alpha_n]_{i_n})=0$.

We are ready to state the WDVV equation.
\begin{prop}[\cite{FWY}]\label{prop:wdvv}
\begin{align*}
    &\sum I_{d_1}([\alpha_{1}]_{i_1}, [\alpha_{2}]_{i_2}, \prod_{j\in S_1} [\alpha_{j}]_{i_j}, \tilde{T}_{i,k}) I_{d_2}(\tilde{T}_{i,k}^{\vee}, [\alpha_{3}]_{i_3}, [\alpha_{4}]_{i_4}, \prod_{j\in S_2} [\alpha_{j}]_{i_j})\\
    =&\sum I_{d_1}([\alpha_{1}]_{i_1}, [\alpha_{3}]_{i_3}, \prod_{j\in S_1}[\alpha_{j}]_{i_j}, \tilde{T}_{i,k}) I_{d_2}(\tilde{T}_{i,k}^{\vee}, [\alpha_{2}]_{i_2}, [\alpha_{4}]_{i_4}, \prod_{j\in S_2} [\alpha_{j}]_{i_j}),
\end{align*}
where each sum takes over all $d_1+d_2=d$, all indices $i,k$ of basis, and $S_1,S_2$ disjoint sets with $S_1 \cup S_2=\{5, ..., n\}$. Also, the $\prod$ symbol makes each factor as a separate insertion, instead of multiplying them up.
\end{prop}

Next, using the above WDVV equation, we can compute relative invariants in the following form:
\begin{equation}\label{eqn:toBdetappB}
\left\langle\vec{\mu}_0|\emptyset|\vec{\mu}_{\infty}\right\rangle
\end{equation}
where $\vec{\mu}_0=(a, b_1, ..., b_n)$ and $\vec{\mu}_\infty=(d)$ are two partitions of $d>0$ such that $b_1, ..., b_n$ are negative. We can rewrite \eqref{eqn:toBdetappB} as
\[I_d([\boldsymbol{0}]_a,[\boldsymbol{0}]_{b_1},\cdots,[\boldsymbol{0}]_{b_n},[\boldsymbol{\infty}]_{d}).\]

We have the following explicit formula.
\begin{prop}
\begin{equation}\label{eqn:fmappB}
I_d([\boldsymbol{0}]_a,[\boldsymbol{0}]_{b_1},\cdots,[\boldsymbol{0}]_{b_n},[\boldsymbol{\infty}]_{d})=d^{n-1}.
\end{equation}
\end{prop}
\begin{proof}
We will prove it by induction on $n$. If $n=0$, then $a=d$. In this case, the corresponding genus zero relative moduli space has only one point with automorphism group $\mu_d$. So
\[I_d([\boldsymbol{0}]_d,[\boldsymbol{\infty}]_d)=d^{-1}.\]
Let us assume that \eqref{eqn:fmappB} holds for all $n\leq m$. We need to show that \eqref{eqn:fmappB} also holds for $n=m+1$. We will
use the WDVV equation of Proposition \ref{prop:wdvv} with the following $m+4$ insertions:
\[[\boldsymbol{0}]_a,[\boldsymbol{\infty}]_{d},[H]_0,[\boldsymbol{\infty}]_{-b_{k+1}},[\boldsymbol{0}]_{b_1},[\boldsymbol{0}]_{b_2},\cdots,[\boldsymbol{0}]_{b_m}.\]
In this case, the WDVV equation becomes
\begin{align*}
    &\sum I_{d_1}([\boldsymbol{0}]_{a}, [\boldsymbol{\infty}]_{d}, \prod_{j\in S_1} [\boldsymbol{0}]_{b_j}, \tilde{T}_{i,k}) I_{d_2}(\tilde{T}_{i,k}^{\vee}, [H]_0, [\boldsymbol{\infty}]_{-b_{m+1}}, \prod_{j\in S_2} [\boldsymbol{0}]_{b_j})\\
    =&\sum I_{d_1}([\boldsymbol{0}]_{a}, [H]_{0}, \prod_{j\in S_1}[\boldsymbol{0}]_{b_j}, \tilde{T}_{i,k}) I_{d_2}(\tilde{T}_{i,k}^{\vee}, [\boldsymbol{\infty}]_{d}, [\boldsymbol{\infty}]_{-b_{m+1}}, \prod_{j\in S_2} [\boldsymbol{0}]_{b_j})
\end{align*}
where each sum takes over all $d_1+d_2=d-b_{m+1}$, all indices $i,k$ of basis, and $S_1,S_2$ disjoint sets with $S_1 \cup S_2=\{1, ..., m\}$. Let us first analyse the LHS of the equation.

If $S_2$ is not empty, then in order to make the total sum of contact orders over $0$ in $I_{d_2}$ non-negative, $\tilde{T}_{i,k}^{\vee}$ must be of the form $[\boldsymbol{0}]_{c}$ with $c>0$. Then using the fact that the total sum of contact orders over $0$ and $\infty$ are the same, we may deduce that $c=-\sum_{j\in S_2}b_j-b_{m+1}$. So $\tilde{T}_{i,k}=[\boldsymbol{0}]_{-c}$ and $d_1=d$, $d_2=-b_{m+1}$. Then by induction and divisor equation, we have
\begin{align*}
I_{d_1}([\boldsymbol{0}]_{a}, [\boldsymbol{\infty}]_{d}, \prod_{j\in S_1} [\boldsymbol{0}]_{b_j}, [\boldsymbol{0}]_{-c})&=d^{|S_1|},\\
I_{d_2}([\boldsymbol{0}]_c, [H]_0, [\boldsymbol{\infty}]_{-b_{m+1}}, \prod_{j\in S_2} [\boldsymbol{0}]_{b_j})&=(-b_{m+1})^{|S_2|}
\end{align*}
where $|S_1|$ and $|S_2|$ are the numbers of elements in $S_1$ and $S_2$. The different ways of choosing the sets $S_1,S_2$ give the combinatorial number $m\choose |S_2|$. So when $S_2$ is not empty, the total contribution to the LHS of WDVV equation is
\[\sum_{|S_2|>0}{m\choose |S_2|}d^{m-|S_2|}(-b_{m+1})^{|S_2|}.\]
Here we have used the fact that $|S_1|+|S_2|=m$.

If $S_2$ is empty, we could similarly deduce that $\tilde{T}_{i,k}^{\vee}$ must be $[\boldsymbol{\infty}]_{b_{m+1}}$ or $[\boldsymbol{0}]_{-b_{m+1}}$. If $\tilde{T}_{i,k}^{\vee}=[\boldsymbol{\infty}]_{b_{m+1}}$, then $I_{d_2}=0$ by dimension reason. If $\tilde{T}_{i,k}^{\vee}=[\boldsymbol{0}]_{-b_{m+1}}$, then the corresponding term on the LHS of WDVV equation becomes
\[I_d([\boldsymbol{0}]_a,[\boldsymbol{0}]_{b_1},\cdots,[\boldsymbol{0}]_{b_{m+1}},[\boldsymbol{\infty}]_{d})I_{d_{m+1}}([\boldsymbol{0}]_{-b_{m+1}},[H]_0, [\boldsymbol{\infty}]_{-b_{m+1}}).\]
By divisor equation, the second term $I_{d_{m+1}}$ in the above product is $1$.
So we may conclude that the total sum on the LHS of WDVV equation is
\begin{equation}\label{eqn:LHSappB}
I_d([\boldsymbol{0}]_a,[\boldsymbol{0}]_{b_1},\cdots,[\boldsymbol{0}]_{b_{m+1}},[\boldsymbol{\infty}]_{d})+\sum_{|S_2|>0}{m\choose |S_2|}d^{m-|S_2|}(-b_{m+1})^{|S_2|}.
\end{equation}

A similar analysis to the RHS of WDVV equation also shows that 
only one term survives in the total summation:
\[I_{d-b_{m+1}}([\boldsymbol{0}]_{a}, [H]_{0},[\boldsymbol{0}]_{b_1},\cdots,[\boldsymbol{0}]_{b_{m}},[\boldsymbol{\infty}]_{d-b_{m+1}})I_0([\boldsymbol{\infty}]_{-d+b_{m+1}},[\boldsymbol{\infty}]_{d},[\boldsymbol{\infty}]_{-b_{m+1}}).\]
The second term in the above product equals to $1$ since there is only one negative insertion (see \cite[Example 5.5]{FWY}). As for the first term, it equals to 
\begin{equation}\label{eqn:RHSappB}
(d-b_{m+1})^m
\end{equation}
by induction and divisor equation. Now LHS=RHS becomes \eqref{eqn:LHSappB}=\eqref{eqn:RHSappB}. So
\[I_d([\boldsymbol{0}]_a,[\boldsymbol{0}]_{b_1},\cdots,[\boldsymbol{0}]_{b_{m+1}},[\boldsymbol{\infty}]_{d})=d^m.\]

\end{proof}

\bibliography{universal-BIB}
\bibliographystyle{amsxport}
\end{document}